\documentclass[twoside]{article}
\usepackage[a4paper,left=30mm,right=30mm,top=30mm,bottom=30mm]{geometry}

\usepackage{fancyhdr}
\newcommand{\helv}{\fontfamily{phv}\fontseries{b}\fontsize{9}{11}\selectfont}
\usepackage{amssymb}
\usepackage{amsfonts}
\usepackage{graphicx}
\usepackage{amsmath}
\usepackage{latexsym}
\usepackage{amssymb}
\usepackage{pstricks}
\usepackage{amsmath}
\usepackage{graphics}
\usepackage{mathrsfs}
\usepackage{fancyhdr}
\usepackage{times}

\newtheorem{theorem}{Theorem}[section]
\newtheorem{lemma}[theorem]{Lemma}
\newtheorem{definition}[theorem]{Definition}

\newtheorem{remark}[theorem]{Remark}

\numberwithin{equation}{section}

\def\div{\mathop{\rm div}}

\newcommand{\om}{\Omega}
\newcommand{\pom}{\partial\Omega}
\newcommand{\epal}{{\mathcal E}(\overline{\om})}

\newcommand{\dom}{{\rm d}\Omega}
\newcommand{\ds}{\mbox{ }ds}
\newcommand{\dt}{\mbox{ }dt}

\newcommand{\dc}{\|}
\newcommand{\dz}{\bigl(}
\newcommand{\bdz}{\bigl(\!\bigl(}

\newcommand{\intnt}{\int_0^T}

\newcommand{\iom}{\int_\om}

\newcommand{\nt}{(0,T)}

\newcommand{\ogpu}{{{\mathcal G}_{\mathcal P}}_{\bfu}}

\newcommand{\omnt}{\om\times\nt}

\newcommand{\on}{{\cal N}}

\newcommand{\sh}{\bigr)}

\newcommand{\svajd}{\bigr)\!\bigr)}

\newcommand{\sumkjinf}{\sum_{k=1}^{\infty}}
\newcommand{\os}{{\cal S}}

\newcommand{\tnt}{t\in(0,T)}

\newcommand{\unu}{{\bfu}_0}

\newcommand{\bfb}{\mbox{\boldmath{$b$}}}
\newcommand{\bfx}{\mbox{\boldmath{$x$}}}
\newcommand{\bfn}{\mbox{\boldmath{$n$}}}
\newcommand{\bff}{\mbox{\boldmath{$f$}}}

\newcommand{\bfg}{\mbox{\boldmath{$g$}}}
\newcommand{\bfG}{\mbox{\boldmath{$G$}}}

\newcommand{\bfu}{\mbox{\boldmath{$u$}}}
\newcommand{\bfw}{\mbox{\boldmath{$w$}}}
\newcommand{\bfv}{\mbox{\boldmath{$v$}}}

\newcommand{\bfU}{\mbox{\boldmath{$U$}}}

\newcommand{\bfphi}{\mbox{\boldmath{$\phi$}}}
\newcommand{\bfPhi}{\mbox{\boldmath{$\Phi$}}}
\newcommand{\bfpsi}{\mbox{\boldmath{$\psi$}}}
\newcommand{\bftheta}{\mbox{\boldmath{$\theta$}}}
\newcommand{\bfvartheta}{\mbox{\boldmath{$\vartheta$}}}
\newcommand{\bfPsi}{\mbox{\boldmath{$\Psi$}}}

\newcommand{\bfsigma}{\mbox{\boldmath{$\sigma$}}}

\newcommand{\bfvarphi}{\mbox{\boldmath{$\varphi$}}}

\newcommand{\bfomega}{\mbox{\boldmath{$\omega$}}}
\newcommand{\bfzero}{{\bf 0}}
\newcounter{constants}
\newcommand{\cn}[2]{ \addtocounter{constants}{1} \newcounter{c#1#2}
\setcounter{c#1#2}{\value{constants}} c_{\arabic{c#1#2}} }
\newcommand{\cc}[2]{c_{\arabic{c#1#2}}}

\title{\fontfamily{phv}\selectfont{\LARGE{\bfseries{Solutions
to the Navier--Stokes Equations with Mixed Boundary Conditions in
Two-Dimensional Bounded Domains}}}}
\date{}
\author{\fontfamily{phv}\selectfont{\large{\mdseries{Michal Bene\v s\footnote{Department of Mathematics, Faculty of Civil Engineering, Czech Technical University in Prague, Th\'{a}kurova 7, 166 29 Prague 6, Czech Republic, E-mail: benes@mat.fsv.cvut.cz}
\; and
Petr Ku\v{c}era\footnote{Department of Mathematics, Faculty of Civil Engineering, Czech Technical University in Prague, Th\'{a}kurova 7, 166 29 Prague 6, Czech Republic, E-mail: kucera@mat.fsv.cvut.cz}}}}}

\begin{document}
\maketitle

\noindent\textbf{Abstract:}   In this paper we consider the system of the
non--steady Navier--Stokes equations with mixed boundary conditions.
We study the existence and uniqueness of a solution of this system.
We define Banach spaces $X$ and $Y$, respectively, to be the space
of ``possible'' solutions of this problem and the space of its data.
We define the operator $\on:X\to Y$ and formulate our problem in
terms of operator equations. Let $\bfu\in X$ and $\ogpu: X\to Y$ be
the Frechet derivative of $\on$ at $\bfu$. We prove that $\ogpu$ is
one-to-one and onto $Y$. Consequently, suppose that the system is
solvable with some given data (the initial velocity and the right
 hand side). Then there
exists a unique solution of this system for data which are small
perturbations of the previous ones.
Next result proved in the Appendix of this paper
is $W^{2,2}$- regularity of solutions of steady Stokes system with mixed
boundary condition for sufficiently smooth data.



\newenvironment{proof}[1][Proof.]{\begin{trivlist}
\item[\hskip \labelsep {\bfseries #1}]}{\end{trivlist}}

\newcommand{\newtext}[1]{{\color{blue} #1}}

\section{Introduction}

The Navier–Stokes equations have been usually solved with the Dirichlet boundary
condition. This theory is elaborated in many papers in which there were proved,
e.g., the results on the global in time existence of weak solutions, uniqueness
of weak solutions in an appropriate function space, global in time existence
of strong solutions for sufficiently small initial data and local in time existence
of strong solution for arbitrary data. However, the Dirichlet boundary condition
is not natural in some situations, e.g. in a finite channel flow model.
The Dirichlet boundary
condition can be used on the fixed wall and on the input of the channel, but it
cannot be prescribed on the output. The reason is the output velocity dependence
on the flow in the channel which is not known in advance. Some authors, dealing mostly with numerical methods,
use either the condition
\begin{equation}\label{eq1}
\nu{{\partial \bfu}\over{\partial\bfn}} - {\mathcal P}\bfn\, =\, \bfzero.
\end{equation}
or
\begin{eqnarray}\label{eq1a}
- {\mathcal P} \bfn +  {\nu\over 2}\,(\nabla\bfu + \nabla\bfu^T)\cdot \bfn &=& {\bf0}
\end{eqnarray}
on the output of the boundary (see e.g. \cite{Glo} or \cite{Ran}). Another possibility is to introduce
mixed boundary conditions by prescribing the homogeneous Dirichlet conditions on the fixed wall and boundary
conditions (\ref{eq1}) or (\ref{eq1a}) on the input and the output of the channel. The latter conditions do not exclude the
possibility of backward flows which could eventually bring an uncontrollable amount of kinetic energy back
to the channel. Consequently the energy inequality known from the Navier-Stokes equations with the Dirichlet
boundary condition or another equivalent a priori estimate of a weak solution cannot be derived for the Navier-Stokes
equations problem with the latter boundary conditions.  Due to
this fact, the question of the global in time existence of a weak
solution of this problem is still open.

Some qualitative properties of the Navier--Stokes equations with these
boundary conditions are studied in
\cite{KraNeu1,KraNeu2,KraNeu5,KuSkpj}. In
\cite{KraNeu1}--\cite{KraNeu5}, Kra\v cmar \& Neustupa prescribed an
additional condition on the output (which bounds the kinetic energy
of an eventual backward flow) and formulated steady and evolutionary
Navier--Stokes problems by means of appropriate variational
inequalities. In \cite{KuSkpj}, Ku\v cera \& Skal\' ak proved the
local--in--time existence of a strong solution of the nonsteady
Navier--Stokes problem with boundary condition (\ref{eq5}) on the
part of the boundary. In this paper, we study the same problem and
we prove the global--in--time existence and uniqueness of a strong
solution in a small neighbourhood of another known solution.

Let $\Omega$ be a bounded domain in $\mathbb{R}^2$, $\Omega \in
C^{0,1}$  and let $\Gamma_D$, $\Gamma_N$ be open disjoint subsets of
$\pom$ (not necessarily connected) such that $\Gamma_D\neq\emptyset$
and the $\partial\om\smallsetminus(\Gamma_D\cup\Gamma_N)$ is a
finite set. The domain $\om$ represents a channel system filled up
by a moving fluid, $\Gamma_D$ is a fixed wall and $\Gamma_N$
represents the input and output (free-stream surfaces) of the
channel. It is assumed that in/outflow pipe segments extend as
straight pipes. All portions of $\Gamma_N$ are taken to be flat and
the boundary $\Gamma_N$ and rigid boundary $\Gamma_D$ form a right
angle at each point $A\in
\partial\om\smallsetminus(\Gamma_D\cup\Gamma_N)$ (i.e., at the point
in which the boundary conditions change their type) (cf. Fig.
\ref{channel}). Moreover, we assume that all parts of $\Gamma_D$ are
smooth (of class $C^{\infty}$). Let $T\in(0,\infty)$, $T$ is
supposed to be fixed value throughout the paper.

The classical formulation of the problem we are going to study is as
follows:
\begin{eqnarray}
{{\partial \bfu}\over{\partial t}} - \nu\Delta \bfu +
(\bfu\cdot\nabla)\bfu + \nabla \mathcal{P} &=& \bff \quad \mbox{in}\
\omnt, \label{eq2} \\ [2pt] \div\bfu
&=& 0\quad \, \mbox{in}\ \omnt, \label{eq3} \\
[4pt]  \bfu &=& {\bf 0} \quad \, \mbox{on}\ \Gamma_D\times\nt,
\label{eq4} \\
[4pt] - \mathcal{P} \bfn + \nu{{\partial \bfu}\over{\partial{\bfn}}}
&=& {\bf0} \quad \,
\mbox{on}\ \Gamma_N\times\nt, \label{eq5} \\
[4pt] \bfu(0) &=& \bfu_0 \;\; \mbox{ in}\ \om. \label{eq6}
\end{eqnarray}
Functions $\bfu$, $\mathcal{P}$, $\bff$, $\bfu_0$ are smooth enough,
$\bfu = (u_1,u_2)$ is velocity, $\mathcal{P}$ represents pressure,
$\nu$ denotes the viscosity, $\bff$ is a body force and $\bfn=
(n_1,n_2)$ is an outer normal vector. $\bfu_0$ describes an initial
velocity and the compatibility condition $\bfu_0=\bf0$ on $\Gamma_D$
holds. The problem (\ref{eq2})--(\ref{eq6}) is called the nonsteady
Navier--Stokes problem with the mixed boundary conditions. For
simplicity we suppose that $\nu=1$ throughout the paper.

We also comment on the problem, in which (\ref{eq3})--(\ref{eq6})
hold and (\ref{eq2}) is replaced with the equation
\begin{equation}\label{eq7}
{{\partial \bfu}\over{\partial t}} - \Delta \bfu + \nabla
\mathcal{P} = \bff    \quad \mbox{in } \omnt .
\end{equation}
The problem (\ref{eq3})--(\ref{eq6}) and (\ref{eq7}) is called the
nonsteady Stokes problem with the mixed boundary conditions.

\begin{figure}[h]
\centering
\includegraphics[width=7.0cm]{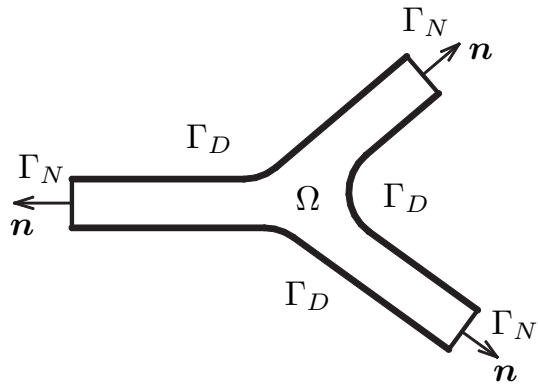}\qquad
  \caption{The domain $\om$
represents a channel filled up by a moving fluid.}\label{channel}
\end{figure}

Let us present an outline of the paper. We start with the
definitions of some function spaces and the spaces of solutions (the
space $X$) and data (the space $Y$) of the problem in Section 2. In
Section 3 we present some auxiliary results of Stokes and
Navier--Stokes problem. We set the problem in the form of an
operator equation. In section 4 we present the main result of the
paper based on the well known Local Diffeomorphism Theorem, i.e. the
local existence and uniqueness result for the related Navier--Stokes
equations with the mixed boundary conditions.  In Appendix A we
prove the regularity of the solution to the steady Stokes problem
with the mixed boundary conditions. We use ideas from
Kozlov~\emph{et al.}~\cite{KozMazRoss}.

We shall denote by $c$ a generic constant, i.e.~a constant whose
value may change from one line to the next one. Numbered constants
$\cn01,\cn02,\dots$ will have fixed values throughout the paper.

\section{Definition of some function spaces}
Let
\begin{displaymath}
{\mathcal{E}}(\overline{\Omega}):=\left\{\bfu\in{\mathcal
C}^\infty(\overline{\Omega})^2; \, \textmd{div}\,\bfu = 0, \,
\textmd{supp}\, \bfu  \cap \overline{\Gamma_D} = \emptyset \right\}.
\end{displaymath}
Let the linear space $V_{\kappa}$ and $L^{2}_{\kappa}$, respectively, be closures of $\epal$ in
the norm of $W^{1,2}(\om)^2$ and $L^2(\Omega)^2$.
Note, that $V_{\kappa}$ and $L^{2}_{\kappa}$ are
closed subspaces of $W^{1,2}(\om)^2$ and $L^2(\om)^2$. The scalar
product in $L^{2}_{\kappa}$ is the same as in $L^2(\Omega)^2$ and we
denote it by $\dz.\,,\,.\sh$. In $V_{\kappa}$, we use the scalar
product
\begin{displaymath}
\bdz\bfPhi,\bfPsi\svajd\ :=\ \iom\nabla\bfPhi\cdot\nabla\bfPsi\;
\dom,
\end{displaymath}
which is equivalent to the scalar product in $W^{1,2}(\Omega)^2$.
Most of papers solving Navier-Stokes equations with homogeneous Dirichlet boundary conditions define
the Hilbert spaces $V$ and $H$. Sometimes these spaces are denoted also by
$W^{1,2}_{0,\sigma}$ and $L^2_{\sigma}(\om)$. The Hilbert spaces $V_{\kappa}$ and $L^{2}_{\kappa}$
defined in this paper play corresponding role as V and H,respectively, but are not the same. To distinguish
them we use symbol $\kappa$.
\smallskip

Let
\begin{equation}\label{eq10b}
\mathcal{D} := \{\bfw\in V_{\kappa};\,\mbox{there exists }\bff\in
L^2_{\kappa} \mbox{ such that }\bdz\bfw,\bfv\svajd =
\dz\bff,\bfv\sh\mbox{ for every } \bfv \in V_{\kappa}\}.
\end{equation}
Let $\bff_i$ and $\bfw_i$, $i=1,2$, are corresponding functions via
(\ref{eq10b}). Denote the scalar product
$\bdz.,.\svajd_{\mathcal{D}}$ on $\mathcal{D}$ such that
\begin{equation}\label{scalar_D}
\bdz\bfw_1,\bfw_2\svajd_{\mathcal{D}} = \dz\bff_1,\bff_2\sh.
\end{equation}
Note that $\mathcal{D}$ is the Hilbert space with the scalar product
defined by \eqref{scalar_D}.

 Let $\bfw\in\mathcal{D}$ and $\bff$ is a corresponding function via (\ref{eq10b}).
It is proved in Appendix \ref{section_stationary} that there exists $q\in L^2(\om)$ such that couple
$(\bfw, q)$ is a solution of the steady Stokes system with mixed boundary conditions
(system (\ref{stationary S 2D})--(\ref{neumann S})).

Bilinear form $\bdz.,.\svajd$ is $V_{\kappa}$-eliptic since
all functions $\bfphi\in V_{\kappa}$ have zero traces on $\Gamma_D$, and
$\Gamma_D$ is nonempty open subset of $\partial\om$. Hence it can be shown as in
\cite[Chapter I., Paragraph 2.6]{Te} that there exist
functions $\bfphi_1,\bfphi_2,\ldots,\bfphi_k, \ldots \in V_{\kappa}
 \subset L^2_{\kappa}$ and real positive numbers $\lambda_1,\lambda_2,
 \ldots\lambda_k,\ldots\to\infty$ for $k\to\infty$, such that
\begin{displaymath}
\bdz\bfphi_k,\bfv\svajd = \lambda_k\dz\bfphi_k,\bfv\sh
\end{displaymath}
for every $\bfv \in V_{\kappa}$. $\bfphi_1,\bfphi_2,\dots$ is a
system which is complete in both $L^2_{\kappa}$ and $V_{\kappa}$,
orthonormal in $L^2_{\kappa}$ and orthogonal in $V_{\kappa}$. It is
easy to see that this system is orthogonal and complete in
$\mathcal{D}$, too. Further
\begin{equation}\label{eq11b}
L^2_{\kappa} = \Bigl\{\bfv;\mbox{ } \bfv = \sumkjinf
a_k\bfphi_k,\mbox{ } a_k\in  \mathbb{R} \mbox{ and }\sumkjinf
a_k^2<\infty\Bigr\},
\end{equation}
\begin{equation}\label{eq12b}
V_{\kappa} = \Bigl\{\bfv;\mbox{ }\bfv = \sumkjinf a_k\bfphi_k,\mbox{
} a_k\in \mathbb{R} \mbox{ and }\sumkjinf \lambda_k
a_k^2<\infty\Bigr\}
\end{equation}
and
\begin{equation}\label{eq13b}
\mathcal{D} = \Bigl\{\bfv;\mbox{ }\bfv = \sumkjinf
a_k\bfphi_k,\mbox{ } a_k\in \mathbb{R}\mbox{ and
}\sumkjinf\lambda_k^2 a_k^2<\infty\Bigr\}.
\end{equation}
Using (\ref{eq11b})--(\ref{eq13b}) we get
\begin{equation}\label{eq13c}
\mathcal{D}\hookrightarrow\hookrightarrow V_{\kappa}.
\end{equation}
In Appendix \ref{section_stationary} we prove the following
embedding
\begin{equation}\label{eq8}
\mathcal{D} \hookrightarrow  W^{2,2}(\om).
\end{equation}
Further we introduce the following Banach spaces
\begin{displaymath}
X := \left\{\bfvarphi;  \; \bfvarphi \in L^2(0,T;\,\mathcal{D}), \;
\bfvarphi' \in L^2(0,T;\,L^{2}_{\kappa})\right\}
\end{displaymath}
and
\begin{displaymath}
Y := \left\{ [\bfg;\,{\bfomega}]; \;  \bfg \in
L^2(0,T;\,L^{2}_{\kappa}), \; {\bfomega} \in V_{\kappa} \right\},
\end{displaymath}
respectively, equipped with the norms
\begin{displaymath}
\|\bfvarphi\|_X := \dc\bfvarphi\dc_{L^2(0,T;\,\mathcal{D})} +
\dc\bfvarphi'\dc_{L^2(0,T;\,L^{2}_{\kappa})}
\end{displaymath}
and
\begin{displaymath}
\|[\bfg;\,{\bfomega}]\|_Y := \|\bfg\|_{L^2(0,T,L^{2}_{\kappa})} +
\dc{\bfomega}\dc_{V_{\kappa}}.
\end{displaymath}

We denote zero elements of $X$ and $Y$
by $\bfzero_X$ and $\bfzero_Y$, respectively. Let us present some properties of the space $X$ which will be used
later. It is easy to see that $\bfvarphi\in X$ if and only if
\begin{eqnarray*}
\bfvarphi(t) = \sum_{k=1}^\infty\vartheta_k(t)\bfphi_k
\end{eqnarray*}
for almost every $\tnt$ and
\begin{eqnarray}\label{eq13ef}
\sum_{k=1}^\infty\int_0^T\bigl(\lambda_k^2\vartheta_k^2(t) + \vartheta^{'\,2}_k(t)\bigr)\dt<\infty.
\end{eqnarray}
Therefore
\begin{eqnarray*}
\frac{d}{\dt}\dc\bfvarphi(t)\dc^2_{V_{\kappa}} = 2\,\sum_{k=1}^\infty\lambda_k\vartheta_k(t)\vartheta'_k(t)
\end{eqnarray*}
and
\begin{eqnarray*}
\frac{d}{\dt}\dc\bfvarphi(t)\dc^2_{L^2_\kappa} = 2\,\sum_{k=1}^\infty\vartheta_k(t)\vartheta'_k(t)
\end{eqnarray*}
for almost every  $\tnt$. Using (\ref{eq13ef}) we obtain
\begin{equation}\label{eq13d}
\frac{d}{\dt}\dc\bfvarphi(t)\dc^2_{L^2_\kappa}\in L^1([0,T])
\end{equation}
and
\begin{equation}\label{eq13ef1}
\frac{d}{\dt}\dc\bfvarphi(t)\dc^2_{V_{\kappa}}\in L^1([0,T]).
\end{equation}
The fact that $\bfvarphi \in L^2(0,T;\,\mathcal{D})\hookrightarrow L^2(0,T;\,V_{\kappa})$ and (\ref{eq13ef1}) imply

\begin{equation}\label{eq13e}
X\hookrightarrow L^{\infty}(0,T;\,V_{\kappa})
\end{equation}
and
\begin{displaymath}
X \subset {\mathfrak C}(0,T;V_{\kappa}).
\end{displaymath}

 Using the embeddings \eqref{eq8} and \cite[Theorem 5.8.2]{KuJoFu} we obtain the embeddings
\begin{equation}\label{eq13ee}
\mathcal{D} \hookrightarrow
W^{2,2}(\om)\hookrightarrow\hookrightarrow
W^{1,6}(\om)^2\hookrightarrow\hookrightarrow L^2(\om)^2.
\end{equation}
(Note that by (\ref{eq8}) $\mathcal{D}\hookrightarrow\hookrightarrow
W^{1,p}(\om)^2$ for every $p>1$, but the embedding
$\mathcal{D}\hookrightarrow\hookrightarrow W^{1,6}(\om)^2$ is
sufficient for our aim now). By \cite[Chapter III, Theorem 2.1.]{Te}
and (\ref{eq13ee}) we get
\begin{equation*}
X\hookrightarrow\hookrightarrow L^2(0,T;\,W^{1,6}(\om)^2).
\end{equation*}
This embedding, (\ref{eq13e}), \cite[Theorem 5.8.2]{KuJoFu}  and the
interpolation between the spaces $V_{\kappa}$ and $W^{1,6}(\om)^2$
yield  the embeddings
\begin{equation}\label{eq13g}
X\hookrightarrow\hookrightarrow
L^4(0,T;\,W^{1,3}(\om)^2)\hookrightarrow L^4(0,T;\,L^q(\om)^2)
\end{equation}
for $1\le q<\infty$.

\section{The nonstationary Stokes and Navier--Stokes equations with the mixed boundary conditions}

Let us start this section with the definition of a generalized
solution to the linearized problem.
\begin{definition}
Let $\bff\in L^2(0,T;\,L^2_{\kappa})$ and $\unu\in V_{\kappa}$. Then
$\bfu$ is a generalized solution of the Stokes problem (\ref{eq3})--
(\ref{eq7}) with the right hand side $\bff$ and the initial
condition $\unu$ if and only if
\begin{equation}\label{eq10e}
\dz \bfu'(t),\bfv\sh +\bdz \bfu(t),\bfv\svajd = \dz \bff(t),\bfv\sh
\end{equation}
for every $\bfv \in V_{\kappa}$ and for almost every $\tnt$ and
\begin{equation}\label{eq10f}
\bfu(0) = \bfu_0.
\end{equation}
\end{definition}
\begin{definition}
The operator $\mathcal{S}: X\to Y$ is defined by
\begin{displaymath}
\os(\bfu) := \Bigl[\dz \bfu',\,.\sh + \bdz
\bfu,.\svajd;\,\bfu(0)\Bigr].
\end{displaymath}
\end{definition}

\begin{remark} Let $\bfu\in X$ and $[\bff;\,\unu]\in Y$.
Then $\bfu$ is the generalized solution of the Stokes problem
(\ref{eq3})-- (\ref{eq7}) with the right hand side $\bff$ and the
initial condition $\unu$ if and only if $\mathcal{S}(\bfu) =
[\bff;\,\unu]$.
\end{remark}

It is obvious that $\mathcal{S}$ is the continuous operator. In the
following theorem we prove that $\mathcal{S}$ is the one-to-one
operator and onto $Y$.

\begin{theorem}\label{stokes}
Let $\bff\in L^2(0,T;\,L^2_{\kappa})$, $\unu\in V_{\kappa}$. There
exists the unique generalized solution $\bfu \in X$ of the Stokes
problem with the mixed boundary conditions and with data $\bff$ and
$\unu$. Moreover, the following estimate holds
\begin{equation}\label{eq11a}
\dc\bfu\dc_{L^2(0,T;\,\mathcal{D})} +
\dc\bfu\dc_{L^{\infty}(0,T;\,V_{\kappa})} +
\dc\bfu'\dc_{L^2(0,T;\,L^2_{\kappa})} \leq \; \cc01
\Bigl(\dc\bff\dc_{L^2(0,T;\,L^2_{\kappa})} +
\dc\unu\dc_{V_{\kappa}}\Bigr).
\end{equation}
\end{theorem}
\begin{proof} Since $\bff\in L^2(0,T;\,L^2_{\kappa})$ and $\unu\in V_{\kappa}$, we have
\begin{equation}\label{eq12a}
\bff = \sumkjinf \mu_k(t)\bfphi_k,\quad\quad \unu =\sumkjinf
a_k\bfphi_k,
\end{equation}
where
\begin{equation}\label{eq13a}
\sumkjinf\intnt \mu_k^2(t)\dt + \sumkjinf a_k^2<\infty.
\end{equation}
(Meaning of (\ref{eq12a}) is that $\bff = \lim\limits_{n\to\infty}\sum_{k=1}^n\mu_k(t)\bfphi_k$ in $L^2(0,T;\,L^2_{\kappa})$ and $\unu = \lim\limits_{n\to\infty}\sum_{k=1}^n a_k\bfphi_k$ in $V_{\kappa}$.) Let $\vartheta_k$ be a solution of the ordinary differential
equation
\begin{equation}\label{eq14}
\vartheta'_k(t) + \lambda_k\vartheta_k(t) = \mu_k(t)
\end{equation}
(which holds for almost every $t \in (0,T)$) with the initial
condition
\begin{equation}\label{eq15}
\vartheta_k(0) = a_k
\end{equation}
for $k = 1,2,\dots$ Then
\begin{equation}\nonumber
\vartheta_k(t) = \int_0^t e^{\lambda_k(s-t)}\mu_k(s)\ds + a_k
e^{-\lambda_k t}
\end{equation}
 for every $\tnt$.
Hence $\vartheta_k\in W^{1,2}((0,t))$. Multiplying (\ref{eq14}) by
$2\vartheta_k'$ and integrating over $(0,t)$ we get
\begin{multline}
 2\int_0^t{\vartheta_k'}^2(s) \ds + \lambda_k\vartheta_k^2(t) =
\lambda_k{\vartheta_k}^2(0) + 2\int_0^t \mu_k(s){\vartheta_k}'(s)\ds \\
\leq \lambda_k\vartheta_k^2(0) + \int_0^t {{\vartheta_k}'}^2(s)\ds +
\int_0^t \mu_k^2(s)\ds
\end{multline}
for $k = 1,2,\dots$ and for every $\tnt$ and therefore
\begin{eqnarray}\label{eq16}
&\displaystyle\int_0^t{\vartheta_k'}^2(s) \ds +
\lambda_k\vartheta_k^2(t)\leq \lambda_k\vartheta_k^2(0) + \int_0^t
\mu_k^2(s)\ds.
\end{eqnarray}
Thus (\ref{eq16}) yields
\begin{multline}\label{eq17}
\sumkjinf\int_0^t{\vartheta_k'}^2(s)\ds +
\sumkjinf\lambda_k\vartheta_k^2(t)\leq
\sumkjinf\intnt{\vartheta_k'}^2(s)\ds +
\sumkjinf\lambda_k\vartheta_k^2(t) \\
 \leq 2\sumkjinf\lambda_k\vartheta_k^2(0) + 2\sumkjinf\intnt\mu_k^2(s)\ds
\end{multline}
for every $\tnt$ (remind that $k$ doesn't depend on $t$) and
therefore we get
\begin{equation}\label{eq18}
\bfu = \sumkjinf\vartheta_k(t)\bfphi_k\in
L^{\infty}(0,T;\,V_{\kappa}),\quad\bfu'\in L^2(0,T;\,L^2_{\kappa})
\end{equation}
and the generalized solution $\bfu$ satisfies the inequality
\begin{equation}\label{eq19}
\dc\bfu\|_{L^{\infty}(0,T;\,V_{\kappa})} +
\dc\bfu'\|_{L^2(0,T;\,L^2_{\kappa})}\leq
2\dc\bff\|_{L^2(0,T;\,L^2_{\kappa})} + 2\dc\unu\|_{V_{\kappa}}.
\end{equation}

(\ref{eq14}) yields also inequalities
\begin{equation}\nonumber
\lambda_k^2\vartheta_k^2(t)\leq 2\mu_k^2(t) + 2{\vartheta'}_k^2(t)
\end{equation}
for every $k = 1,2,\dots$ and for almost every $\tnt$. Therefore we
get \begin{displaymath} \sumkjinf \lambda_k^2
\int_0^t\vartheta_k^2(s)\ds\leq \sumkjinf \lambda_k^2
\int_0^T\vartheta_k^2(s)\ds\leq 2\sumkjinf\intnt\mu_k^2(s)\ds +
2\sumkjinf\intnt{\vartheta'}_k^2(s)\ds.
\end{displaymath}
The last inequality and (\ref{eq17}) yield
\begin{equation}\label{eq19a}
\sumkjinf\lambda_k^2\int_0^t\vartheta_k^2(s)\ds\leq
\sumkjinf\lambda_k^2\intnt\vartheta_k^2(s)\ds\leq
6\sumkjinf\intnt\mu_k^2(s)\ds + 4\sumkjinf\lambda_k\vartheta_k^2(0)
\end{equation}
for every $\tnt$. Therefore one obtains
\begin{displaymath}
\bfu\in L^2(0,T;\,\mathcal{D}).
\end{displaymath}
Moreover, \eqref{eq19a} implies the estimate
\begin{displaymath}
\|\bfu\|_{L^2(0,T;\,\mathcal{D})}\leq \cc02
\Bigl(\dc\bff\|_{L^2(0,T;\,L^2_{\kappa})} +
\dc\unu\|_{V_{\kappa}}\Bigr).
\end{displaymath}
The last inequality and (\ref{eq19}) imply (\ref{eq11a}). It is easy
to see that $\bfu \in X$ and
\begin{displaymath}
\dz \bfu'(t),\bfv\sh + \bdz \bfu(t),\bfv\svajd = \dz \bff(t),\bfv\sh
\end{displaymath}
for every $\bfv\in V_{\kappa}$ and for almost every $\tnt$ and that
\begin{displaymath}
\bfu(0)=\unu.
\end{displaymath}

The existence of the unique generalized solution $\bfu$ for given data $\bff$ and $\unu$ will now be proven.
Suppose that $\bfu_A, \bfu_B\in X$ are solutions of this problem for given data $\bff$ and $\unu$. We
prove that $\bfu_A = \bfu_B$.

Denote $\bfw = \bfu_A - \bfu_B$. Then
\begin{equation}\label{eq19b}
\dz \bfw'(t),\bfv\sh +\bdz \bfw(t),\bfv\svajd = 0
\end{equation}
for every $\bfv \in V_{\kappa}$ and for almost every $\tnt$ and
\begin{equation}\label{eq19c}
\bfw(0) = \bfzero.
\end{equation}
Multiplying (\ref{eq19b}) by $\bfw(t)$, integrating over $\nt$ and using (\ref{eq13d}) and (\ref{eq19c}) we obtain
\begin{equation}\label{eq19d}
\dc\bfw(T)\dc^2_{L^2_\kappa} + \intnt\dc\bfw(t)\dc^2_{V_{\kappa}}\dt = 0.
\end{equation}
Therefore we get $\bfw=\bfzero_X$ and consequently $\bfu_A = \bfu_B$. This
completes the proof.
\end{proof}
\bigskip

If $\bftheta,\,\bfpsi,\,\bfphi \in V_{\kappa}$, then
$b(\bftheta,\bfpsi,\bfphi )$ denotes the trilinear form
\begin{equation}\label{eq28}
b(\bftheta,\bfpsi,\bfphi ) =
\iom\theta_j{{\partial\psi_i}\over{\partial x_j}}\phi _i  \;\dom.
\end{equation}
The summation convention is used for repeated indices.

\begin{remark}\label{remark1} Let $\bftheta,\,\bfpsi\in \mathcal{D}$ then
 $b(\bftheta,\bfpsi,.)\in L^2_{\kappa}$. If
$\bfu,\bfw\in L^2(0,T;\,\mathcal{D})\cap
L^{\infty}(0,T;\,V_{\kappa})$ then $b(\bfu,\bfw,.) =
b(\bfu(t),\bfw(t),.)\in L^2(0,T;\,L^2_{\kappa})$. Moreover,
\phantom{$\cn03$}
\begin{equation}\label{eq28a}
\dc b(\bfu,\bfw,. )\|_{L^2(0,T;\,L^2_{\kappa})}\leq  \; \cc03 \;
\dc\bfu\|_{X}\dc\bfw\|_{X}.
\end{equation}
\end{remark}

\bigskip

Now we set up a generalized formulation of the Navier--Stokes
problem.
\begin{definition}\label{navier}
Let $\bff\in L^{2}(0,T;\,L^2_{\kappa})$, $\unu\in V_{\kappa}$. Then
$\bfu\in L^2(0,T;\,\mathcal{D})\cap L^{\infty}(0,T;\,V_{\kappa})$
with $\bfu'\in L^2(0,T;\,L^2_{\kappa})$ is called a generalized
solution of the problem (\ref{eq2}) -- (\ref{eq6}) on $(0,T)$ (a
generalized solution of the Navier--Stokes problem with the mixed
boundary conditions) with data $\bff$ and $\unu$ if and only if
\begin{equation}\label{eq29}
\dz\bfu'(t),\bfv\sh + \bdz\bfu(t),\bfv\svajd
+b(\bfu(t),\bfu(t),\bfv) = \dz\bff(t),\bfv\sh
\end{equation}
holds for all $\bfv \in V_{\kappa}$ and for almost every $t \in
(0,T)$, and
\begin{equation}\label{eq30}
\bfu(0) = \unu.
\end{equation}
\end{definition}

\bigskip

Define the operator $\mathcal{N}: X \rightarrow Y$ given by
\begin{equation}\label{mestac N}
\mathcal{N}(\bfu) := \left[ \dz\bfu'(t),.\sh + \bdz\bfu(t),.\svajd +
b(\bfu(t),\bfu(t),.);\, \bfu (0)  \right].
\end{equation}
\begin{remark}
Let $\bfu \in X$ and $[\bff;\, \bfu_0] \in Y$.  The generalized
problem can now be treated as one operator equation
\begin{displaymath}
\mathcal{N}(\bfu)=[\bff;\, \bfu_0].
\end{displaymath}
\end{remark}

\bigskip

Let $\bfu$ be a fixed point in $X$. Let $\mathcal{B}_u: X
\rightarrow Y$ be a linear operator defined by
\begin{equation}\label{nestac B}
\mathcal{B}_u (\bfw) := \left[  b(\bfu,\bfw,.) + b(\bfw,\bfu,.);\,
{\bf0}  \right].
\end{equation}
\begin{theorem}
Let $\bfu$ be  some arbitrary fixed element in $X$. The operator
$\mathcal{G}_u: X \rightarrow Y$  given by
\begin{equation}\label{nestac G}
\mathcal{G}_u(\bfw) := \mathcal{S}(\bfw) + \mathcal{B}_u(\bfw)
\end{equation}
is the Fr\'{e}chet derivative of $\mathcal{N}$ at the point $\bfu$,
$\mathcal{G}_u \in \mathcal{C}(X \times X, Y)$ and $\mathcal{N} \in
\mathcal{C}^1(X, Y)$.
\end{theorem}
\begin{proof} Since
\begin{equation}\label{eqpom5}
\|  \mathcal{N}(\bfu + \bfw) - \mathcal{N}(\bfu) -
\mathcal{G}_u(\bfw)  \|_Y = \|[b(\bfw,\bfw,.);\, {\bf0}]\|_Y
\end{equation}
and (\ref{eq28a}) yields the estimate
\begin{equation}\label{eqpom6}
\|[b(\bfw,\bfw,.);\, {\bf0}]\|_Y \leq  \;  \cc03 \; \|\bfw\|^2_X,
\end{equation}
we get
\begin{displaymath}
\lim_{\|\bfw\|_X \rightarrow 0}  \frac{\| \mathcal{N}(\bfu + \bfw) -
\mathcal{N}(\bfu) - \mathcal{G}_u(\bfw)  \|_Y}{\|\bfw\|_X} = 0  .
\end{displaymath}
$\mathcal{N} \in \mathcal{C}^1(X, Y)$ and the smoothness
$\mathcal{G}_u \in \mathcal{C}(X \times X, Y)$ is obvious. The proof
is complete.
\end{proof}

\section{Main result}
We can now state the main result of the paper.
\begin{theorem}[{\bf Main result}]\label{maintheorem1}
Let $\bfu\in X$, $[\bff;\,\bfu_0]\in Y$ and $\bfu$ be the
generalized solution of the Navier--Stokes initial--boundary value
problem with right hand side $\bff$ and initial velocity $\bfu_0$,
formulated by the operator equation ${\cal
N}(\bfu)=[\bff;\,\bfu_0]$. Then there exist open sets ${\cal
U}\subset X$ and ${\cal V}\subset Y$ such that $\bfu\in {\cal U}$,
$[\bff;\,\bfu_0]\in{\cal V}$ and for every
$[\widetilde{\bff};\,\widetilde{\bfu}_0]\in{\cal V}$ there exists
unique $\widetilde{\bfu}\in{\cal U}$ which is a generalized solution of the
Navier-Stokes initial--boundary value problem with right hand side $\widetilde{\bff}$
and initial velocity $\widetilde{\bfu_0}$, formulated by the
operator equation ${\cal
N}(\widetilde{\bfu})=[\widetilde{\bff};\,\widetilde{\bfu}_0]$.
Conversely, let $\widetilde{\bfu}\in{\cal U}$. Then there exists unique
$[\widetilde{\bff};\,\widetilde{\bfu}_0]\in{\cal V}$ such that $\widetilde{\bfu}\in{\cal U}$
is a generalized solution of the
Navier-Stokes initial--boundary value problem with right hand side $\widetilde{\bff}$
and initial velocity $\widetilde{\bfu_0}$.
\end{theorem}

We prepare the following lemmas and propositions to prove our main
result which is postponed to the end of this section.

\begin{lemma}
Let $\bfu \in X$. Then $\mathcal{B}_u$ is a compact operator from
$X$ into $Y$.
\end{lemma}
\begin{proof} Let $\left\{ \bfw_n \right\} \subset X$ be a
bounded sequence. We prove that there exists a subsequence $\left\{
\bfw_{n_k} \right\}$ of $\left\{ \bfw_n \right\}$ and $\bfw\in X$
such that $ b(\bfu,\bfw_{n_k},.)+b(\bfw_{n_k},\bfu,.)\to
b(\bfu,\bfw,.)+b(\bfw,\bfu,.) $ in $L^2(0,T;\,L^2_{\kappa})$.

Since $X$ is reflexive and $\left\{ \bfw_n \right\}$ is bounded in
$X$, there exists a subsequence $\left\{ \bfw_{n_k} \right\}$ and
 $\bfw \in X$ such that
\begin{equation*}
\hspace{7mm}\bfw_{n_k} \rightarrow \bfw \qquad \textmd{ weakly in }
X.
\end{equation*}
Using (\ref{eq13g}) we obtain
\begin{eqnarray}\label{conv3}
&&\bfw_{n_k}\rightarrow\bfw \qquad \textmd{ in }
L^4(0,T;\,W^{1,3}(\om)^2)
\end{eqnarray}
and
\begin{displaymath}
\bfu \in L^4(0,T;\,W^{1,3}(\om)^2).
\end{displaymath}
Since
\begin{eqnarray}\nonumber
\| (\bfu\cdot\nabla)(\bfw_{n_k}- \bfw)\|^2_{L^2(0,T;\,L^2_\kappa)}
&\leq& c \; \int^T_0 \| \bfu \|^2_{L^6{(\Omega})^2} \|\nabla (
\bfw_{n_k}- \bfw) \|^2_{L^3{(\Omega})^2} \;{\rm d}t \\
\nonumber &\leq& c \; \left( \int^T_0 \| \bfu \|^4_{L^6{(\Omega})^2}
\;{\rm d}t \right)^{1/2} \left(
\int^T_0 \| \bfw_{n_k}-\bfw \|^4_{\,W^{1,3}(\om)^2)} \;{\rm d}t \right)^{1/2}\\
&\leq&   c \; \| \bfu \|^2_{L^4(0,T;\,W^{1,3}(\om)^2)} \|
\bfw_{n_k}- \bfw \|^2_{L^4(0,T;\,W^{1,3}(\om)^2)}, \nonumber
\end{eqnarray}
(\ref{conv3}) implies that
\begin{equation}\label{conv1}
\| (\bfu\cdot\nabla)(\bfw_{n_k}- \bfw) \|_{L^2(0,T;\,L^2_{\kappa})}
\rightarrow 0.
\end{equation}
The same way, the estimate
\begin{eqnarray*}
\| ((\bfw_{n_k}- \bfw)\cdot\nabla)
\bfu\|^2_{L^2(0,T;\,L^2_{\kappa})} &\leq& c \; \int^T_0 \|
\bfw_{n_k}-\bfw \|^2_{L^6{(\Omega})^2} \| \nabla \bfu
\|^2_{L^3{(\Omega})^2}
\; {\rm d}t  \\
&\leq& c \; \left( \int^T_0  \| \bfw_{n_k}-\bfw
\|^4_{L^6{(\Omega})^2} \; {\rm d}t \right)^{1/2} \left( \int^T_0 \|
\nabla \bfu \|^4_{L^3{(\Omega})^2}
\; {\rm d}t \right)^{1/2}  \\
&\leq& c \;   \| \bfw_{n_k}- \bfw \|^2_{L^4(0,T;\,W^{1,3}(\om)^2)}
\| \bfu \|^2_{L^4(0,T;\,W^{1,3}(\om)^2}
\end{eqnarray*}
and (\ref{conv3}) imply
\begin{equation}\label{conv2}
\| ((\bfw_{n_k}- \bfw)\cdot\nabla) \bfu\|_{L^2(0,T;\,L^2_{\kappa})}
\rightarrow 0.
\end{equation}
Therefore
\begin{displaymath}
\|(\bfu\cdot\nabla)(\bfw_{n_k}- \bfw)+ ((\bfw_{n_k}-
\bfw)\cdot\nabla) \bfu\|_{L^2(0,T;\,L^2_{\kappa})} \rightarrow 0.
\end{displaymath}
It is easy to see that
\begin{displaymath}
\mathcal{B}_{\bfu}(\bfw_{n_k})\to\mathcal{B}_{\bfu}(\bfw) \quad
\text{ in } Y.
\end{displaymath}
 The proof is complete.
\end{proof}

\begin{lemma}\label{injective Bu}
Let $\bfu \in X$.  $\mathcal{G}_u$ is an injective operator from $X$
to $Y$.
\end{lemma}
\begin{proof} Suppose that $\mathcal{G}_u (\bfw)=\textbf{0}_Y$. Then
\begin{displaymath}
\dz\bfw'(t),\bfv\sh + \bdz\bfw(t),\bfv\svajd
+b(\bfw(t),\bfu(t),\bfv)+b(\bfu(t),\bfw(t),\bfv)\,=\,0
\end{displaymath}
holds for all $\bfv \in V_{\kappa}$ and every $t \in (0,T)$ and
$\bfw(0)\equiv \textbf{0}$. Hence
\begin{eqnarray*}
\frac{1}{2}\frac{d}{\dt}\|\bfw(t)\|^2_{L^2_{\kappa}} + \|\nabla
\bfw(t)\|^2_{L^2_{\kappa}} &\leq&
|b(\bfu(t),\bfw(t),\bfw(t))|+|b(\bfw(t),\bfu(t),\bfw(t))|  \\
&\leq&  \|\nabla\bfw(t)\|^{7/4}_{L^2(\Omega)^2}
\|\bfw(t)\|^{1/4}_{L^2_{\kappa}} \|\bfu(t)\|_{L^4(\Omega)^2} \\
&& \lefteqn{+\|\nabla \bfw(t) \|^{3/2}_{L^2(\Omega)^2}
\|\bfw(t)\|^{1/2}_{L^2_{\kappa}} \| \nabla \bfu(t)\|_{L^2(\Omega)^2} }\\
&\leq&\frac{1}{2}\|\nabla\bfw(t)\|^2_{L^2(\Omega)^2} + c \,
\|\bfw(t)\|^2_{L^2_{\kappa}} \|\bfu(t)\|^8_{L^4(\Omega)^2} \\  &&
\lefteqn{+ \frac{1}{2}
 \|\nabla \bfw(t)
\|^2_{L^2(\Omega)^2}  + c \; \|\bfw(t)\|^2_{L^2_{\kappa}} \|
\nabla \bfu(t)\|^4_{L^2(\Omega)^2}}\\
\end{eqnarray*}
and therefore
\begin{equation}
\frac{d}{\dt}\|\bfw(t)\|^2_{L^2_{\kappa}} \leq c \; \| \bfw(t)
\|^2_{L^2_{\kappa}} \left(\|\bfu(t)\|^8_{L^4(\Omega)^2}+\|\nabla
\bfu(t)\|^4_{L^2(\Omega)^2}\right).
\end{equation}
It is easy to see that $\bigl(\|\bfu(t)\|^8_{L^4(\Omega)^2}+\|\nabla
\bfu(t)\|^4_{L^2(\Omega)^2}\bigr)\in L^1((0,T))$ and $\|\bfw(0)\|^2_{L^2_{\kappa}}=0$.
Using Gronwall's lemma, we obtain $\bfw\equiv\bfzero$. The proof is complete.
\end{proof}

We remind the well known Local Diffeomorphism Theorem.
\begin{theorem}\label{Local diffeomorphism theorem}
Let $\mathcal{X}$ and $\mathcal{Y}$  be Banach spaces, $f$ be a
mapping from $\mathcal{X}$ into  $\mathcal{Y}$ belonging to
$\mathcal{C}^1$ in some neighborhood $V$ of a point $\bfu$. If
$f'(\bfu): X \rightarrow Y$ is one-to-one and onto $\mathcal{Y}$ and
continuous, then there exists a neighborhood $U$ of point $\bfu$, $U
\subset V$ and a neighborhood $W$ of point $f(\bfu)$, $W \subset
\mathcal{Y}$ such that $f$ is one-to-one from $V$ onto $W$.
\end{theorem}

The following theorem plays crucial role in the proof of our main
result (see \cite[Theorem 5.5.F]{Tay}).
\begin{theorem}\label{theorem10}
Let $\mathcal{X}$, $\mathcal{Y}$ be Banach spaces, $\mathcal{F}_o$
be a one-to-one operator from $\mathcal{X}$ onto $\mathcal{Y}$,
$\mathcal{F}_c$ be a compact linear operator $\mathcal{X}$ into
$\mathcal{Y}$. The following statements are equivalent:

 (a) $\mathcal{F}_o + \mathcal{F}_c$ is an injective operator

 (b) $\mathcal{F}_o + \mathcal{F}_c$ is an operator onto
 $\mathcal{Y}$ .
\end{theorem}

\begin{proof}[ {\bf Proof of Theorem \ref{maintheorem1}} ]
$\mathcal{G}_u$ is a sum of the operators $\mathcal{S}$ and
$\mathcal{B}_u$. Note that $\mathcal{S}: X \rightarrow Y$ is the
one-to-one operator and onto $Y$ and
 $\mathcal{B}_u: X \rightarrow Y$  is a compact operator.
 Moreover $\mathcal{G}_u = \mathcal{S}+\mathcal{B}_u$ is a one-to-one mapping.
 Using Theorem \ref{theorem10} and Lemma \ref{injective Bu} we get that $\mathcal{G}_u$ is
 a one-to-one operator and onto $Y$. The continuity of $\mathcal{G}_u$ is
 obvious. Finally, Theorem \ref{Local diffeomorphism theorem} yields the
 assertion.
\end{proof}

\begin{remark}[Uniqueness of the operator $\mathcal{N}$]
Let $\bfu_1$, $\bfu_2$ $\in$ $X$, $\mathcal{N}(\bfu_1) =
\mathcal{N}(\bfu_2)$, then $\bfu_1 = \bfu_2$.
\end{remark}
\begin{proof}[Sketch of the proof] Denote $\bfw=\bfu_1 - \bfu_2$ then
\begin{displaymath}
[\dz\bfw'(t),\bfv\sh + \dz\bfw(t),\bfv\svajd
+\bfb(\bfw(t),\bfu_2(t),\bfv)+\bfb(\bfu_1(t),\bfw(t),\bfv);\,\bfw(0)
]= \textbf{0}_Y.
\end{displaymath}
Using procedure similar to that in the proof of Lemma \ref{injective
Bu} we get $\bfw =\bfu_1 - \bfu_2 = \textbf{0}_X$.
\end{proof}

\appendix

\section{The steady Stokes problem with mixed boundary
conditions}\label{section_stationary}

In this appendix we prove some results on the regularity of the
steady Stokes system with mixed boundary conditions. We further use
these results in order to prove the continuous embedding
\eqref{eq8}.

Let us consider the boundary value problem
\begin{eqnarray}
-  \Delta \bfvartheta  +  \nabla q &=& \bfsigma \quad \textmd{in}
\quad \Omega,\label{stationary S 2D}\\
\nabla \cdot \bfvartheta &=& 0 \quad \, \textmd{in} \quad
\Omega,\label{rovnice kontinuity S}\\
\bfvartheta &=& \textbf{0}  \quad
\textmd{on} \quad \Gamma_D,\label{dirichlet S}\\
- q\bfn +  \frac{\partial \bfvartheta}{\partial \bfn} &=& \textbf{0}
\quad \textmd{on} \quad \Gamma_N,\label{neumann S}
\end{eqnarray}
where $\bfvartheta=(\vartheta_1,\vartheta_2)$ denotes the velocity
field, $q$ is the associated pressure and $\bfsigma \in
L^2(\Omega)^2$ is a body force.

A pair $(\bfvartheta,q)\in V_{\kappa}\times L^2(\Omega)$ is called
the weak solution of the problem (\ref{stationary S
2D})--(\ref{neumann S}) if $\bfvartheta$ satisfies
\begin{equation}\label{eq8a}
((\bfvartheta,\bfv))=\langle\bfsigma,\bfv\rangle
\end{equation}
for all $\bfv \in V_{\kappa}$ and  $\bfvartheta$ and $q$ satisfy the
equation (\ref{stationary S 2D}) in $\Omega$ in the sense of
distributions. Here $\langle \cdot , \cdot \rangle$ denotes the
duality between $V_{\kappa}$ and $V_{\kappa}^*$.
 Since bilinear form $((. ,.))$ is $V_{\kappa}$-elliptic
there exists a unique $\bfvartheta\in V_{\kappa}$ such that
(\ref{eq8a}) holds. By \cite[Chapter I, Proposition 1.2.]{Te} there
exists $q\in  L^2(\Omega)$ such that equation (\ref{stationary S
2D}) is satisfied in the sense of distributions and
\begin{equation}\label{eq8b}
\dc\bfvartheta\dc_{V_{\kappa}} + \dc q\dc_{L^2(\Omega)}\leq\cn12\dc
\bfsigma\dc_{{L^2(\Omega)^2}},
\end{equation}
where $\cc12 = \cc12(\om)$.
 Our aim is to
prove the next theorem, which immediately implies (\ref{eq8}).

\begin{theorem}\label{maintheorap}
Let $\bfsigma \in L^2(\Omega)^2$ and $(\bfvartheta,q)$ be a weak
solution of \eqref{stationary S 2D}--\eqref{neumann S} with the
right hand side $\bfsigma$. Then $(\bfvartheta,q)$ belongs to
$W^{2,2}(\Omega)^2 \times W^{1,2}(\Omega)$. Moreover,
\begin{eqnarray}\label{eq23a}
\|\bfvartheta\|_{W^{2,2}(\om)^2} + \|q\|_{W^{1,2}(\om)} \leq \cn07
\, \|\bfsigma\|_{L^2(\Omega)^2}
\end{eqnarray}
with some constant $\cc07 = \cc07(\om)$.
\end{theorem}

Note that the system (\ref{stationary S 2D})--(\ref{neumann S})
represents an elliptic boundary value problem in the sense of Agmon,
Douglis and Nirenberg \cite[Chapter I.1.]{AgDoNi} and \cite[Part
III., Chapter 1., Section 1.4.]{GaRaRoTu}. The general questions
about solvability (Fredholm's property) and regularity of solutions
to the linear elliptic boundary value problems in domains with
corners are solved for instance in \cite{Kon} by Kondrat'ev, in
\cite{KozMazRoss,KozMazRoss2001} by Kozlov~\textit{ et al.} and in
\cite{KufSan} by Kufner and S\"{a}ndig.

Let $\om_1$, $\om_2$ be arbitrary open sets such that
$\om_1\subset\overline{\om}_1\subset\om_2\subset\overline{\om}_2\subset\om$.
By \cite[Theorem IV.4.1]{Ga} the pair $(\bfvartheta,q)$ belongs
to $W^{2,2}(\Omega_1)^2\times W^{1,2}(\om_1)$ and
\begin{eqnarray*}\label{eq23b}
\|\bfvartheta\|_{W^{2,2}(\om_1)^2} + \|q\|_{W^{1,2}(\om_1)} \leq\;
\cn14 \; (\|\bfsigma\|_{L^2(\Omega_2)^2} + \|q\|_{L^(\om_2)} + \|\bfvartheta\|_{W^{1,2}(\om_2)^2}),
\end{eqnarray*}
where $\cc14 = \cc14(\om_1,\om_2)$. This estimate and (\ref{eq8b}) imply
\begin{eqnarray}\label{eq23b}
\|\bfvartheta\|_{W^{2,2}(\om_1)^2} + \|q\|_{W^{1,2}(\om_1)} \leq\;
\cn04 \; \|\bfsigma\|_{L^2(\Omega)^2},
\end{eqnarray}
where $\cc04 = \cc04(\om_1,\om)$.

In order to show that the solution is locally regular at points $P$
on $\Gamma_N$, we use an appropriate infinitely differentiable
cut--off function, which equals $1$ in a small neighbourhood
$U_{\tau}(P)$ ($U_{\tau}(P)$ denotes the ball of radius $\tau$
centered at the point $P$) and $0$ outside $U_{2\tau}(P)$, whose
values are between $0$ and $1$ in $U_{2\tau}(P)\smallsetminus
U_{\tau}(P)$ and which depends only on the distance from point $P$.
Multiplying equation (\ref{stationary S 2D}) by this cut--off
function and using the assumption that $\Gamma_N$ is a (open) line
segment, we transform the problem (\ref{stationary S
2D})--(\ref{neumann S}) to the system
%
%
\begin{eqnarray}
-  \Delta \bfvartheta'  +  \nabla q' &=& \bfsigma' \quad \textmd{in
} \om_3,
\label{adn1}\\
\nabla \cdot \bfvartheta' &=& \chi'\quad \, \textmd{in
}\om_3,
\label{adn2}\\
- q'\bfn +  \frac{\partial \bfvartheta'}{\partial \bfn} &=&
\textbf{0} \quad \textmd{ on }  \partial\om_3,\label{adn3}
\end{eqnarray}
where $\om_3\subset\om$ is an appropriate smooth domain containing
$U_{2\tau}(P)\cap\Omega$,
$\bfsigma'\in L^2(\om_3)^3$ and $\chi'\in
W^{1,2}(\om_3)$. The new unknown functions $\bfvartheta'$
and $q'$, respectively, coincide with $\bfvartheta$ and $q$ in
$U_{\tau}(P)$. Moreover,
\begin{equation}\label{eqhelpest1}
\dc \bfsigma'\|_{L^2(\om_3)^2}
+\dc\chi'\|_{W^{1,2}(\om_3)}\leq c\,\dc
\bfsigma\|_{L^2(\Omega)^2}.
\end{equation}
(This cut--off function procedure is described in greater detail
e.g.~in \cite[Theorem D.1]{Bocev}.) The system
(\ref{adn1})--(\ref{adn3}) is of the Agmon-Douglis-Nirenberg type
(ADN). Using the regularity theory for elliptic systems in smooth
domains (cf. \cite{AgDoNi}, \cite[Theorem D.1]{Bocev}), we deduce
that any weak solution of $(\bfvartheta',q')$ of
\eqref{adn1}--\eqref{adn3} belongs to $W^{2,2}(\om_3)^2\times
W^{1,2}(\om_3)$ and satisfies the estimate
\begin{eqnarray*}\label{estimate_ADN}
\|\bfvartheta'\|_{W^{2,2}(\om_3)^2} +
\|q'\|_{W^{1,2}(\om_3)} \leq\; c\, (
\|\bfsigma\|_{L^2(\om_3)^2}+\|\chi'\|_{W^{1,2}(\om_3)}+\|\bfvartheta\|_{L^2(\om_3)^2}).
\end{eqnarray*}
This estimate, (\ref{eq8b}) and (\ref{eqhelpest1}) imply that
\begin{eqnarray}\label{estimate_ADN1}
\|\bfvartheta\|_{W^{2,2}(\om_N)^2} +
\|q\|_{W^{1,2}(\om_N)} \leq\; \cn05
\|\bfsigma\|_{L^2(\Omega)^2},
\end{eqnarray}
$\om_N = \om\cap U_{\tau}(P)$, which confirms that the solution $(\bfvartheta,q)$ is regular in the
neighbourhood of point $P$, $\cc05 = \cc05(\om_N)$.

In order to show that the solution is locally regular at points $P$
on $\Gamma_D$, we apply the analogous cut-off function technique
with the only difference that the boundary condition is
\begin{equation}\label{adn4}
\bfvartheta = \bfzero \quad \textmd{ on }  \partial\om_D,
\end{equation}
where $\om_D$ is an appropriate smooth domain in
$U_{2\tau}(P)\cap\Omega$, and we obtain the estimate
\begin{eqnarray}\label{estimate_ADN2}
\|\bfvartheta\|_{W^{2,2}(\om_D)^2} +
\|q\|_{W^{1,2}(\om_D)} \leq\; \cn08
\|\bfsigma\|_{L^2(\Omega)^2},
\end{eqnarray}
where $\cc08 = \cc08(\om_D)$.

\subsection{ Local regularity at the point in which the boundary
conditions change their type}

We have explained that the weak solution $(\bfvartheta,q)$ of
problem of \eqref{stationary S 2D}--\eqref{neumann S} belongs to
$W_{loc}^{2,2}(\Omega)^2\times W_{loc}^{1,2}(\om)$ and satisfies
\eqref{eq23b}. Furthermore, this solution is ``locally regular'' in
the neighborhood of an arbitrary point $P$ in
$\Gamma_D\cup\Gamma_N$. To prove that the solution $(\bfvartheta,q)$
is ``globally regular'', we need to show that it is ``locally
regular'' at the point $A$, where the boundary conditions change the
type (see Fig.~\ref{localization}). Since the complete proof is long
and relatively technical, we sketch its main ideas in the next
subsection.

\subsubsection{Basic ideas of the proof of regularity in a neighbourhood
of the corner points}

We apply the method, developed by Kondrat'ev, Kozlov, Kufner,
M\"arkl, Maz´ya, Oleinik, Orlt, Rosman and S\" andig, whose
principles are explained e.g.~in \cite{Kon}, \cite{KonOle},
\cite{KozMazRoss}, \cite{KozMazRoss2001}, \cite{KufSan},
\cite{MarklSan} and \cite{OrltSan}.

The weak solvability of the problem (\ref{stationary S
2D})--(\ref{neumann S}) is known. It is explained above that in
order to prove the regularity of the weak solution in the whole
domain $\Omega$, it remains to verify the regularity in some
neighbourhood of point $A$. Recall that $A$ is the point on the
boundary, where the boundary conditions change the type. At first we
localize the boundary value problem (\ref{stationary S
2D})--(\ref{neumann S}) in the neighbourhood of $A$ by means of an
appropriate cut-off function $\eta$ (by analogy with the steps
described above). We choose the origin of the coordinate system to
be identical with point $A$ (see Fig. \ref{localization}). Suppose
that the cut--off function $\eta = \eta(|\bfx|) \in
{C}^{\infty}(\mathbb{R}^2)$ satisfies $0 \leq \eta(|\bfx|)\leq 1$
and
\begin{equation}\label{cut_of_function}
\eta(|\bfx|) =  \quad  \left\{
\begin{array}{ccl}
1 & {\rm for} &  |\bfx| < \delta / 2 , \\
0 & {\rm for} &  |\bfx| > \delta.
\end{array} \right.
\end{equation}
(Here $\delta$ is a positive number so small that $A$ is the only
corner point in the circle $\left\{\bfx: |\bfx|\leq\delta
\right\}$.)

Denote $\bfw = \eta\bfvartheta$ and $Q = \eta q$. Let $\mathcal{K}$
be the angle of the size ${\pi/2}$, enclosed by the two
perpendicular tangential vectors to $\partial \Omega$ at point $A$.
Since $(\bfvartheta,q)$ solves equations (\ref{stationary S
2D})--(\ref{rovnice kontinuity S}), $(\bfw,Q)$ satisfies equations
(\ref{eq20a}) and (\ref{eq20b}) in $\mathcal{K}$.

Let $\tilde{S} =\left\{ (\xi,\omega): \xi \in \mathbb{R}, \,
0<\omega<{\pi / 2}  \right\}$ be an infinite strip (see Fig.
\ref{infinite_strip}). By means of the change of coordinates
$(x_1,x_2)\rightarrow (\xi,\omega)$, where $(r,\omega)$ are the
polar coordinates with the origin $A$ and $\xi=\log r$, we transform
the pair $(\bfw,Q)$ to the pair $(\tilde{\bfw},\tilde{Q})$. We shall
see in subsection A.2 that $(\tilde{\bfw},\tilde{Q})$ solves the
equations (\ref{eq21a})--(\ref{eq21c}) (the so called model problem)
in $\tilde{S}$.

Applying the complex Fourier transform (see \cite[Chapter
I, Section I]{KufSan}, with respect to the variable $\xi$, we transform the
pair $(\tilde{\bfw},\tilde{Q})$ (of the variables $\xi$, $\omega$)
to the pair $(\widehat{\bfw},\widehat{Q})$ (depending on $\lambda$,
$\omega$). If we consider $\lambda$ to be fixed then
$(\widehat{\bfw},\widehat{Q})$, satisfy, as functions of only one
variable $\omega$, the system of three ordinary differential
equations (\ref{eq22a})--(\ref{eq22c}) on the interval $(0,\pi/2)$
with parameter $\lambda$ (see equations (\ref{eq22a})--(\ref{eq22c})
in subsection A.2). This system can be written in the form of one
operator equation
$\mathcal{A}(\lambda)(\hat{w}_1,\hat{w}_2,\hat{Q})=(\hat{G}_1,\hat{G}_2,\hat{H})$,
where the mapping
\begin{displaymath}
\lambda\to\mathcal{A}(\lambda):W^{2,2}((0; \pi/2))^2\times
W^{1,2}((0; \pi/2))\to L^2((0; \pi/2))^2\times W^{1,2}((0; \pi/2))
\end{displaymath}
is defined by (\ref{eq24a}) for all $\lambda\in\mathbb{C}$.
\smallskip

\begin{figure}[h]
\centering \includegraphics[width=7cm]{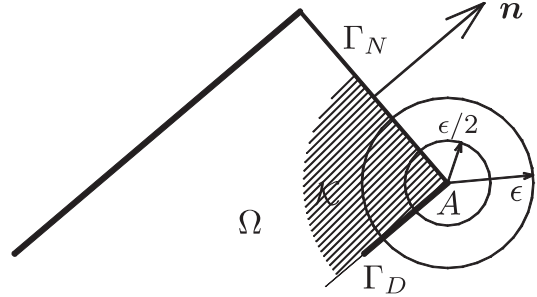}\\
\caption{Localization of the problem near the corner point $A$ where
the boundary conditions change their type.}\label{localization}
\end{figure}

Transforming the boundary conditions to the same way, we receive the
boundary conditions (\ref{eq25a})--(\ref{eq25b}) (for $\omega = 0$)
and (\ref{eq25c})--(\ref{eq25d}) (for $\omega = \pi/2$). Further, we
define certain matrix operators
\begin{gather*}
\mathcal{B}_{DN,1}(\lambda):W^{2,2}((0; \pi/2))^2\times W^{1,2}((0;
\pi/2))\to \mathbb{C}^2, \\ \noalign{\vskip 4pt}
\mathcal{B}_{DN,2}(\lambda):W^{2,2}((0; \pi/2))^2\times W^{1,2}((0;
\pi/2))\to \mathbb{C}^2,
\end{gather*}
associated with the boundary conditions (\ref{eq25a})--(\ref{eq25b})
and (\ref{eq25c})--(\ref{eq25d}), see (\ref{BDN1}) and (\ref{BDN2})
for details. These operators naturally depend on parameter
$\lambda$. Then we put
\begin{displaymath}
\widehat{\mathcal{L}}(\lambda):=\left[\mathcal{A} (\lambda);\,
\mathcal{B}_{DN,1}(\lambda);\,\mathcal{B}_{DN,2}(\lambda)\right]
\end{displaymath}
for $\lambda\in \mathbb{C}$ (see (\ref{operL})). Operator
$\widehat{\mathcal{L}}(\lambda)$ maps $W^{2,2}(0,\pi/2)^2 \times
W^{1,2}(0,\pi/2)$ into $L^{2}(0,\pi/2)^2 \times W^{1,2}(0,\pi/2)
\times \mathbb{C}^2 \times \mathbb{C}^2$.

The parameter dependent operator $\widehat{\mathcal{L}}(\lambda)$ is
a so called pencil operator corresponding to the problem
(\ref{stationary S 2D})--(\ref{neumann S}). Note that it is possible
to define the pencil operator at every boundary point for every
elliptic boundary value problem in the sense of Agnon, Douglis and
Nirenberg. Note further that every generalized steady Stokes system
(with arbitrary type of boundary conditions) is elliptic boundary
value problems in the sense of Agnon, Douglis and Nirenberg. Now we
define its eigenvalues and simple eigenvalues.

\begin{definition}
The complex number $\lambda=\lambda_0$ is an eigenvalue of
$\widehat{\mathcal{L}}(\lambda)$ if there exists a nontrivial
solution $\widehat{\bfU}(.,\lambda_0)\in
\mathscr{D}(\widehat{\mathcal{L}}(\lambda))$ which is holomorphic at
$\lambda_0$, $\widehat{\bfU}(.,\lambda_0)\neq \bf0$, and
$\widehat{\mathcal{L}}(\lambda_0)\widehat{\bfU}(.,\lambda_0)=\bf0$.
$\widehat{\bfU}(.,\lambda_0) = \widehat{\bfU}(\omega,\lambda_0)$ is
an eigenfunction of $\widehat{\mathcal{L}}(\lambda_0)$ with respect
to $\lambda_0$.
\end{definition}

\begin{definition}
Let $\lambda_0$ be an eigenvalue of
$\widehat{\mathcal{L}}(\lambda)$. We say that it is a simple
eigenvalue if
$\widehat{\mathcal{L}}'(\lambda_0)\widehat{\bfU}(.,\lambda_0)=\bf0$
only for $\widehat{\bfU}(.,\lambda_0)=\bf0$.
\end{definition}

Note (see e.g.~in \cite{Kon}, \cite{KonOle}, \cite{KozMazRoss}) that
if $\lambda$ is not an eigenvalue of $\widehat{\mathcal{L}}$, then
operator $\widehat{\mathcal{L}}(\lambda)$ is an isomorphism between
spaces $W^{2,2}(0,\pi/2)^2 \times W^{1,2}(0,\pi/2)$ and
$L^{2}(0,\pi/2)^2 \times W^{1,2}(0,\pi/2) \times \mathbb{C}^2 \times
\mathbb{C}^2$.

The main proposition of this section  (Theorem \ref{maintheorap}) is
based on Theorem \ref{theorem pom 2}. To prove Theorem \ref{theorem
pom 2} we will apply the following theorem which is the simplified
version of Theorems 1.4.3 and 1.4.4 in \cite{KozMazRoss2001}.

\begin{theorem}[Regularity and a priori
estimate]\label{theorem_kozmazross} Let $(\overline{\bfvartheta},\overline{q})\in
W^{1,2}(\Omega)^2\times L^2(\Omega)$ be the weak solution of some
generalized steady Stokes systems with a right hand side
$\overline{\bfsigma}=(\overline{\sigma_1},\overline{\sigma_2},\overline{\sigma_3})\in
L^p(\om)^2\times W^{1,p}(\Omega)$, $p>1$, $A\in\partial\om$. Denote by
$\widehat{\mathcal{B}} = \widehat{\mathcal{B}}(\lambda)$ its
corresponding pencil operator. Then the following propositions hold:
\begin{itemize}
\item  Assume that $\lambda_0$ is the only eigenvalue of
$\widehat{\mathcal{B}}(\lambda)$ in the strip ${\rm Im}\, \lambda\in
(2/p-2,0)$. Suppose additionally that this eigenvalue is simple.
Assume that the lines ${\rm Im}\, \lambda=0$ and ${\rm Im}\,
\lambda=2/p-2$ are free of eigenvalues of the pencil operator
$\widehat{\mathcal{B}}(\lambda)$. Then there exists a cut-off
function $\eta = \eta(r)$ and $\delta>0$ (see (\ref{cut_of_function})) such that
$(\overline{\bfvartheta},\overline{q}) =
(\overline{\bfvartheta}(r,\omega),\overline{q}(r,\omega))$
admits in a neighborhood $\mathcal{O}$ of the
corner point $A$ the asymptotic representation
\begin{equation}\label{asymptotic_expansion}
\eta(r)\left(
\begin{array}{r}
\overline{\bfvartheta}\\
\overline{q}
\end{array}\right)
=c \left(
\begin{array}{r}
\overline{\bfvartheta}_{sing}
\\
\overline{q}_{sing}
\end{array}\right)  + \left(
\begin{array}{r}
\overline{\bfvartheta}_{reg}
\\
\overline{q}_{reg}
\end{array}\right),
\end{equation}
where
$\Bigl(\overline{\bfvartheta}_{reg},\overline{q}_{reg}\Bigr)
\in W^{2,p}(\om_\delta)^2\times W^{1,p}(\om_\delta)$
and $\om_\delta = U_{\delta}(A)\cap\Omega$. Constant $c$ is called
generalized intensity factor and the corresponding singular function
is given by
\begin{equation}\label{vector function}
\left(
\begin{array}{r}
\overline{\bfvartheta}_{sing}
\\
\overline{q}_{sing}
\end{array}\right) = r^{i \lambda}
\left(
\begin{array}{r}
\dot{\bfvartheta}
\\
r^{-1}\dot{q}
\end{array}\right),  \nonumber
\end{equation}
where $(\dot{\bfvartheta},\dot{q}) = (\dot{\bfvartheta}(\omega),\dot{q}(\omega))$ is the
corresponding eigenfunction of $\widehat{\mathcal{B}}(\lambda_0)$.
\item Suppose that the line ${\rm Im}\,
\lambda = 2/p-2$ does not contain eigenvalues of the pencil operator
$\widehat{\mathcal{L}}(\lambda)$ and
 $(\overline{\bfvartheta},\overline{q})\in W^{2,p}(\Omega_\delta)^2\times W^{1,p}(\Omega_\delta)$.
Then
\begin{equation}\label{estimate_general}
\|\overline{\bfvartheta}\|_{W^{2,p}(\om_A)^2}+
\|\overline{q}\|_{W^{1,p}(\om_A)}\leq \cn11
\,\|\overline{\bfsigma}\|_{L^{p}(\Omega)^2},
\end{equation}
where $\om_A = U_{\tau}(A)\cap\Omega$ for some $\tau<\delta/2$ and
$\cc11 = \cc11(\om_\delta,\tau)$.
\end{itemize}
%
\end{theorem}
\begin{remark}
Since (\ref{stationary S 2D})--(\ref{neumann S}) represents an
elliptic boundary value problem in the sense of Agmon, Douglis and
Nirenberg, $\overline{\bfvartheta_1} = \bfvartheta_1$,
$\overline{\bfvartheta_2} = \bfvartheta_2$ and $\overline{q} = q$, we can
apply the previous theorem for our problem.
\end{remark}

\begin{remark}\label{remark_mellin_fourier} Note that we use
Fourier transform instead of Mellin transform used in
\cite{KozMazRoss,KozMazRoss2001,OrltSan}. Consequently, we study the
existence of eigenvalues in the strip ${\rm Im}\, \lambda\in
(-1-\varepsilon,0)$ instead of ${\rm Re}\, \lambda \in
(0,1+\varepsilon)$ for sufficiently small $\varepsilon$. (For
detailed theory of boundary value problems in nonsmooth domains
based on Fourier technique see \cite{Kon,KonOle,KufSan}.)
\end{remark}

We will show (see Remark \ref{roots}) that  only the simple
eigenvalue $\lambda_0=-{\rm i}$ is situated in the strip ${\rm
Im}\lambda\in [-1-\varepsilon,0)$ choosing $\varepsilon>0$
sufficiently small.


\subsubsection{The pencil operator}  Our aim in this subsection is
to derive the pencil operator for our problem. Consider the weak
solution $(\bfvartheta,q)$ of (\ref{stationary S 2D})--(\ref{neumann
S}). Suppose additionally (only in this subsection) that
$(\bfvartheta,q) \in W^{2,2}(\Omega)^2 \times W^{1,2}(\Omega)$.
Choose the origin $O$ at the point $A$ with an angle $\pi/2$ and
multiply the equations (\ref{stationary S 2D})--(\ref{rovnice
kontinuity S}) by the ``cut off function'' $\eta$. Remind $\bfw=\eta
\bfvartheta$ and $Q=\eta q$. Further, denote by $\mathcal{K}$ an
infinite angle with the vertex $O\equiv A$ and size ${\pi / 2}$.
Then we have
\begin{eqnarray}
- \Delta \bfw + \nabla Q &=& \bfg \quad \textmd{in} \quad
\mathcal{K},\label{eq20a}\\
\nabla \cdot \bfw &=& h  \quad \textmd{in} \quad
\mathcal{K},\label{eq20b}
\end{eqnarray}
where
\begin{equation}\label{eq20c}
\bfg =-\bfvartheta\Delta \eta-2\frac{\partial \bfvartheta}{\partial
x_1}\frac{\partial \eta}{\partial x_1}-2 \frac{\partial
\bfvartheta}{\partial x_2}\frac{\partial \eta}{\partial
x_2}+\bfsigma \eta + (\nabla \eta)q,\,\,\,\,\,\, h=\bfvartheta \cdot
(\nabla \eta)
\end{equation}
and $\bfg \in L^2(\mathcal{K})^2$, $h \in W^{1,2}(\mathcal{K})$,
$\bfw \in W^{2,2}(\mathcal{K})^2$, $Q \in W^{1,2}(\mathcal{K})$. The
behavior of $\bfw=\eta\bfvartheta$ and $Q=\eta q$ near $O$
characterizes the regularity of $\bfu$ and $q$ in a neighborhood of
the point $A$.

Under the polar coordinates $(r,\omega)$ the Stokes problem
(\ref{eq20a})--(\ref{eq20b}) becomes
\begin{eqnarray}
- \left( \frac{\partial ^2 \bar{w}_1}{\partial r^2} + \frac{1}{r}
\frac{\partial \bar{w}_1}{\partial r} +  \frac {1}{r^2} \frac
{\partial ^2 \bar{w}_1}{\partial \omega ^2} \right) + \frac{\partial
\bar{Q} }{\partial r}\cos{\omega} -\frac{1}{r}\frac{\partial
\bar{Q}}{\partial \omega} \sin \omega &=& \bar{g}_1 (r,\omega)
,\label{eq20e}\\
- \left( \frac{\partial ^2 \bar{w}_2}{\partial r^2} + \frac{1}{r}
\frac{\partial \bar{w}_2}{\partial r} +  \frac {1}{r^2} \frac
{\partial ^2 \bar{w}_2}{\partial \omega ^2} \right) + \frac{\partial
\bar{Q} }{\partial r}\sin{\omega} + \frac{1}{r}\frac{\partial
\bar{Q}}{\partial \omega} \cos \omega &=& \bar{g}_2 (r,\omega)
,\label{eq20f}\\
\frac{\partial \bar{w}_1}{\partial r}  \cos \omega - \frac{1}{r}
\frac {\partial \bar{w}_1}{\partial \omega} \sin \omega +
\frac{\partial \bar{w}_2}{\partial r}\sin \omega + \frac{1}{r} \frac
{\partial \bar{w}_2}{\partial \omega} \cos \omega
&=&\bar{h}(r,\omega)\label{eq20g}
\end{eqnarray}
that holds in $\bar{S}$, where $\bar{S}=\left\{ (r,\omega):
0<r<\infty, \, 0<\omega<{\pi / 2} \right\}$ is the infinite angle
described in polar coordinates $(r,\omega)$ (see Fig.
\ref{finite_strip}), $\bar{\bfw}(r,\omega)=\bfw(x_1,x_2)$,
$\bar{Q}(r,\omega)=Q(x_1,x_2)$,
$\bar{\bfg}(r,\omega)=\bfg(x_1,x_2)$,
$\bar{h}(r,\omega)=h(x_1,x_2)$.

\begin{figure}[h]
\begin{minipage}[b]{.5\textwidth}
\includegraphics[width=5.9cm]{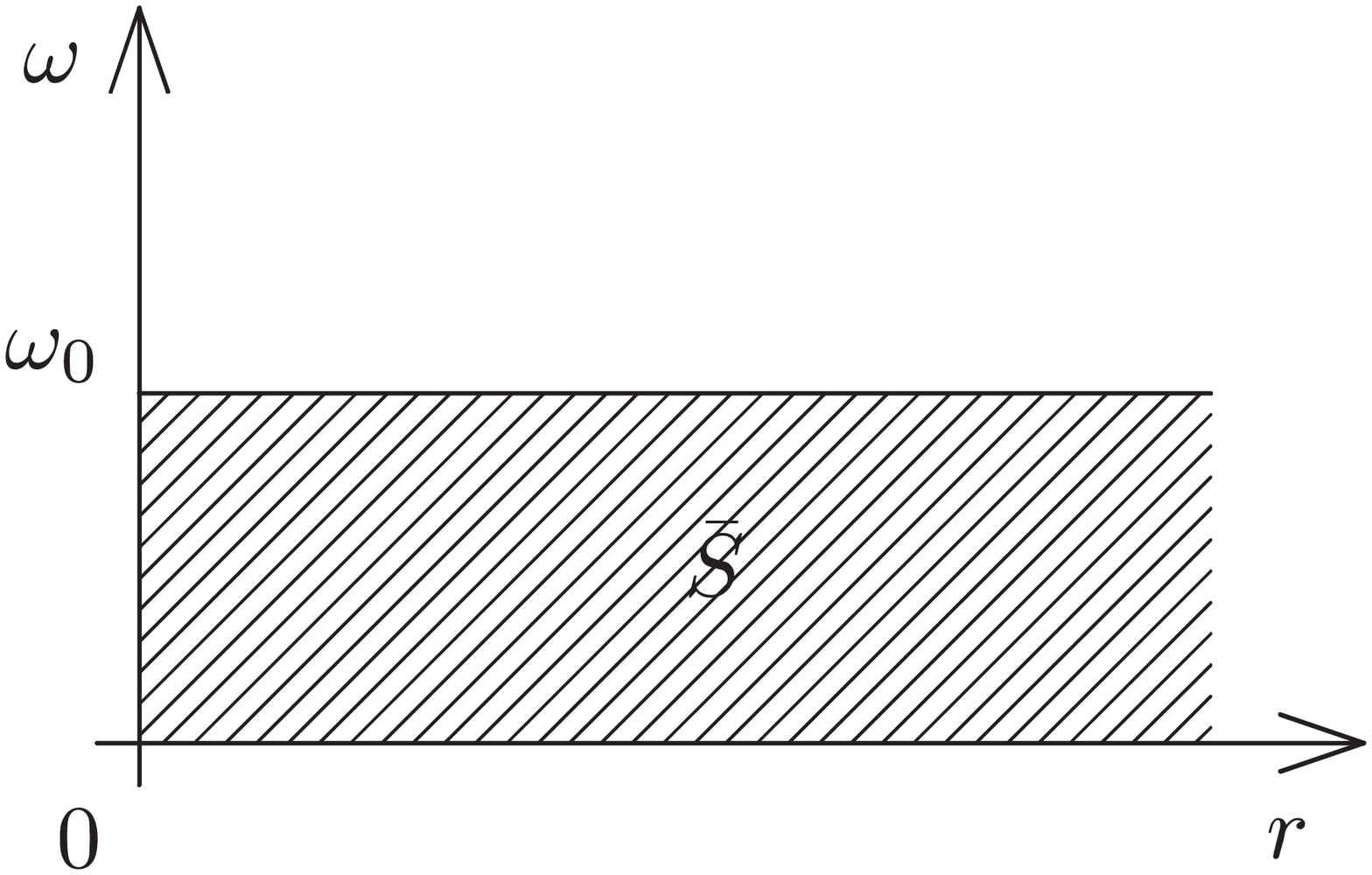}
\caption{The infinite half--strip $\bar{S}$.}\label{finite_strip}
\end{minipage}
\begin{minipage}[b]{.6\textwidth}
\includegraphics[width=7.5cm]{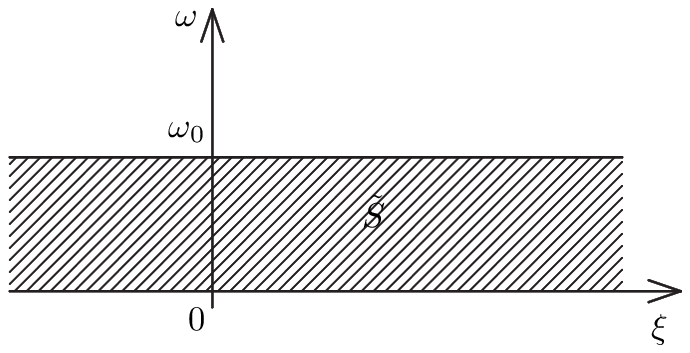}
\caption{The infinite strip $\tilde{S}$.}\label{infinite_strip}
\end{minipage}
\end{figure}

Using the substitution $r=e^\xi$ we get
\begin{eqnarray}
- \left( \frac{\partial ^2 \tilde{w}_1}{\partial \xi^2} + \frac
{\partial ^2 \tilde{w}_1}{\partial \omega ^2} \right) + \left(
\frac{\partial \tilde{Q}}{\partial \xi} - \tilde{Q}\right)
\cos{\omega} -\frac{\partial\tilde{Q}}{\partial \omega} \sin \omega
&=&\widetilde{G}_1(\xi,\omega) ,\label{eq21a}\\
- \left( \frac{\partial ^2 \tilde{w}_2}{\partial \xi^2} + \frac
{\partial ^2 \tilde{w}_2}{\partial \omega ^2} \right) + \left(
\frac{\partial\tilde{Q}}{\partial \xi}-\tilde{Q}\right) \sin{\omega}
+ \frac{\partial\tilde{Q}}{\partial \omega}
\cos\omega &=&\widetilde{G}_2(\xi,\omega) ,\label{eq21b} \\
\frac{\partial \tilde{w}_1}{\partial \xi}  \cos \omega - \frac
{\partial \tilde{w}_1}{\partial \omega} \sin \omega + \frac{\partial
\tilde{w}_2}{\partial \xi} \sin \omega + \frac {\partial
\tilde{w}_2}{\partial \omega} \cos \omega &=&
\widetilde{H}(\xi,\omega) \label{eq21c}
\end{eqnarray}
that holds in $\tilde{S}=\left\{ (\xi,\omega): \xi \in \mathbb{R},
\, 0<\omega<{\pi / 2}  \right\}$ (see Fig. \ref{infinite_strip}),
$\tilde{\bfw}(\xi,\omega)=\bfw(x_1,x_2)$,
$\tilde{Q}(\xi,\omega)=e^{\xi}Q(x_1,x_2)$,
$\tilde{\bfg}(\xi,\omega)=\bfg(x_1,x_2)$,
$\tilde{h}(\xi,\omega)=h(x_1,x_2)$, $\widetilde{\bfG}(\xi,\omega) =
e^{2\xi}\tilde{\bfg}(\xi,\omega)$, $\widetilde{H}(\xi,\omega) =
e^{\xi}\tilde{h}(\xi,\omega)$. Note that $\tilde{\bfw}\in
W^{2,2}({\tilde{S}})^2$, $\tilde{Q} \in W^{1,2}({\tilde{S}})$,
$\widetilde{\bfG} \in L^2({\tilde{S}})^2$, $\widetilde{H}\in
W^{1,2}({\tilde{S}})$. Applying complex Fourier transform with
respect to $\xi$ for suitable $\lambda\in\mathbb{C}$, we get the
following system of three ordinary differential equations depending
on a parameter $\lambda$ with unknown functions $\widehat{w}_1$,
$\widehat{w}_2$ and $\widehat{{Q}}$
\begin{eqnarray}
- \frac {\partial ^2 \widehat{w}_1}{\partial \omega ^2} - (
\textrm{i} \lambda)^2(\widehat{w}_1)+(\textrm{i}\lambda-1
)\widehat{{Q}}\cos \omega-\frac{\partial \widehat{{Q}}}{\partial
\omega} \sin \omega
&=&\widehat{G}_1(\lambda,\omega),\label{eq22a}\\
- \frac {\partial ^2\widehat{w}_2}{\partial \omega ^2}-(\textrm{i}
\lambda)^2(\widehat{w}_2)+(\textrm{i}\lambda-1)\widehat{{Q}}\sin\omega
+\frac{\partial\widehat{Q} }{\partial \omega} \cos \omega &=&
\widehat{G}_2(\lambda,\omega),\label{eq22b}\\
( \textrm{i}\lambda)(\widehat{w}_{1})\cos \omega - \frac {\partial
\widehat{w}_1}{\partial \omega} \sin \omega + ( \textrm{i} \lambda )
(\widehat{w}_2)\sin \omega+\frac{\partial\widehat{w}_2 }{\partial
\omega} \cos \omega&=& \widehat{H}(\lambda,\omega) \label{eq22c}
\end{eqnarray}
that holds in the interval $(0 , \pi/2)$, $\widehat{\bfG}=
\mathcal{F}_{\xi\rightarrow \lambda}(\widetilde{\bfG})$,
 $\widehat{H}=\mathcal{F}_{\xi\rightarrow
\lambda}(\widetilde{H})$,
$\widehat{\bfw}=\mathcal{F}_{\xi\rightarrow \lambda}(\tilde{\bfw})$,
$\widehat{Q}=\mathcal{F}_{\xi\rightarrow \lambda}(\tilde{Q})$.  Note
that for the complex parameter $\lambda$ we have $\widehat{\bfw}\in
W^{2,2}((0;\pi/2))^2$, $\widehat{Q} \in W^{1,2}((0;\pi/2))$,
$\widehat{\bfG} \in L^2((0; \pi/2))^2$, $\widehat{H}\in W^{1,2}((0;
\pi/2))$.

\medskip

Denote by $\mathcal{A}(\lambda):W^{2,2}((0; \pi/2))^2\times
W^{1,2}((0; \pi/2))\to L^2((0; \pi/2))^2\times W^{1,2}((0; \pi/2))$
the matrix operator which corresponds to system
(\ref{eq22a})--(\ref{eq22c}) , i.e.
\begin{equation}\label{eq24a}
\mathcal{A}(\lambda) = \left(
\begin{array}{ccc}
-\frac{\partial}{\partial \omega^2}-(i\lambda)^2&0&(i\lambda-1)\cos\omega-\sin\omega \frac{\partial}{\partial\omega}\\
0&-\frac{\partial}{\partial \omega^2}-(i\lambda)^2&(i\lambda-1)\sin\omega+\cos\omega \frac{\partial}{\partial\omega}\\
(i\lambda)\cos\omega-\sin\omega\frac{\partial}{\partial\omega}&(i\lambda)\sin\omega+\cos\omega\frac{\partial}{\partial\omega}
&0\\
\end{array}\right).
\end{equation}
We considered this operator for all parameter
$\lambda\in\mathbb{C}$.

\medskip

\subsubsection*{The mixed boundary conditions}
``Localizing'' the problem, introducing polar coordinates
$(r,\omega)$ and substituting $r=e^{\xi}$ we get the mixed boundary
conditions \eqref{neumann S} at the point $\omega=0$ and $\omega =
{\pi / 2}$
\begin{eqnarray}
\frac{\partial \tilde{w}_1}{\partial
\omega} (\xi,0)&=& 0, \label{neumann SSm3}\\
-(\tilde{Q}e^{\xi})(\xi,0)+\frac{\partial \tilde{w}_2}{\partial
\omega}(\xi,0)&=& 0, \label{neumann SSm4}\\
\tilde{w}_1(\xi,\pi/2)&=&0, \label{dirichlet SSmx}\\
\tilde{w}_2(\xi,\pi/2)&=&0 \label{dirichlet SSm}
\end{eqnarray}
and using the Fourier transform with respect to $\xi$,
\eqref{neumann SSm3}--\eqref{dirichlet SSm} read
\begin{eqnarray}
 \frac {\partial \widehat{w}_1
}{\partial \omega}  (\lambda,0 )&=&0 ,\label{eq25a}\\
 \frac {\partial  \widehat{w_2}}
{\partial \omega}(\lambda,0)-\widehat{{Q}}
(\lambda, 0)&=&0,\label{eq25b}\\
\widehat{w}_1(\lambda,\pi/2) &=& 0,\label{eq25c}\\
\widehat{w}_2(\lambda,\pi/2) &=& 0.\label{eq25d}
\end{eqnarray}
Denote by $\mathcal{B}_{DN,1}(\lambda)$ the operator of the boundary
conditions of mixed type (\ref{eq25a})--(\ref{eq25d}) written in the
matrix form for $\omega=0$ (Neumann type condition)
\begin{eqnarray}\label{BDN1}
\mathcal{B}_{DN,1}(\lambda)=\left(
\begin{array}{ccc}
\underline{\frac{\partial}{\partial \omega}}\big|_{0}&0& 0\\
0&\underline{\frac{\partial}{\partial \omega}}\big|_{0}&-\underline{1}\big|_{0}\\
\end{array}\right)
\end{eqnarray}
and $\mathcal{B}_{DN,2}(\lambda)$ for $\omega=\pi/2$ (Dirichlet
condition)
\begin{eqnarray}\label{BDN2}
\mathcal{B}_{DN,2}(\lambda)= \left(
\begin{array}{ccc}
\underline{1}\big|_{\frac{\pi}{2}}&0& 0\\
0&\underline{1}\big|_{\frac{\pi}{2}}&0\\
\end{array}\right).
\end{eqnarray}
Remind that $\widehat{\mathcal{L}}(\lambda)$ is the parameter
dependent operator which is defined by
\begin{eqnarray}\label{operL}
\widehat{\mathcal{L}}(\lambda)=\left[\mathcal{A} (\lambda);\,
\mathcal{B}_{1}(\lambda);\,\mathcal{B}_{2}(\lambda)\right].
\end{eqnarray}
$\widehat{\mathcal{L}}(\lambda)$ is considered for all
$\lambda\in\mathbb{C}$ and it corresponds to the problem
(\ref{eq22a})--(\ref{eq22c}) with the boundary conditions
(\ref{eq25a})--(\ref{eq25d}).
\begin{equation}\label{zobrazeniDN}
\widehat{\mathcal{L}}(\lambda): W^{2,2}((0; \pi/2))^2 \times
W^{1,2}((0; \pi/2))
 \rightarrow L^{2}((0; \pi/2))^2 \times W^{1,2}((0; \pi/2)) \times \mathbb{C}^2
\times \mathbb{C}^2.
\end{equation}

\subsubsection{Calculation of the characteristic determinants to the
Stokes flows and a regularity result for the stationary Stokes
problem}

Denote by $[\hat{e}_1;\,\hat{e}_2;\,\hat{e}_p]$ the general solution
of the system (\ref{eq22a})--(\ref{eq22c}) with the vanishing right
hand side, where $\hat{e}_1$, $\hat{e}_2$ stand for $\widehat{w}_1$,
$\widehat{w}_2$ and $\hat{e}_p$ stands for $\widehat{{Q}}$,
respectively. The general solution
$[\hat{e}_1;\,\hat{e}_2;\,\hat{e}_p]$ has the form
\begin{multline}\label{general_solution_1}
\left(\!
\begin{array}{c}
\hat{e}_1\\
\hat{e}_2\\
\hat{e}_q\\
\end{array}\!\right) =
C_1 \left(
\begin{array}{c}
\cos(i \lambda\omega)\\
-\sin(i \lambda \omega)\\
0\\
\end{array}\right)+
C_2 \left(
\begin{array}{c}
\sin(i \lambda\omega)\\
\cos(i \lambda \omega)\\
0\\
\end{array}\right) \\+
C_3 \left(
\begin{array}{c}
-\frac{i\lambda}{2}\cos[({i}\lambda-2)\omega]\\
\sin({i} \lambda \omega)+\frac{{i}\lambda}{2}\sin[({i}\lambda-2)\omega]\\
-2{i}\lambda\cos[({i}\lambda-1)\omega]\\
\end{array}\right)
+C_4 \left(
\begin{array}{c}
\frac{{i}\lambda}{2}\sin[({i}\lambda-2)\omega]\\
\cos({i} \lambda \omega)+\frac{{i}\lambda}{2}\cos[({i}\lambda-2)\omega]\\
2{i}\lambda\sin[({i}\lambda-1)\omega]\\
\end{array}\right)
\end{multline}
for $\lambda \neq 0$ and
\begin{multline}\label{general_solution_2}
\left(\!
\begin{array}{c}
\hat{e}_1\\
\hat{e}_2\\
\hat{e}_q\\
\end{array}\!\right) =
C_1 \left(\!
\begin{array}{c}
\cos(2\omega)\\
\sin(2\omega)-2\omega\\
4\cos\omega\\
\end{array}\!\right)+
C_2 \left(\!
\begin{array}{c}
-\sin(2\omega)-2\omega\\
\cos(2\omega)\\
-4\sin\omega\\
\end{array}\!\right) +
C_3 \left(\!
\begin{array}{c}
1\\
0\\
0\\
\end{array}\!\right)
+C_4 \left(\!
\begin{array}{c}
0\\
1\\
0\\
\end{array}\!\right)
\end{multline}
for $\lambda = 0$. Remark that $\lambda = 0$ is not an eigenvalue of
the pencil $\widehat{\mathcal{L}}(\lambda)$. Recall that every
$\lambda_0 \in \mathbb{C}$ such that ker
$\widehat{\mathcal{L}}(\lambda_0)\neq \left\{ \bf0 \right\}$ is said
to be an eigenvalue of $\widehat{\mathcal{L}}(\lambda)$. The
distribution of the eigenvalues of the operator
$\widehat{\mathcal{L}}(\lambda)$ plays crucial role in the
regularity results of the solution, see Theorem \ref{theorem pom 2}.

Substituting the general solution
(\ref{general_solution_1}) and (\ref{general_solution_2}) into the
corresponding boundary conditions \eqref{eq25a}--\eqref{eq25d}
 we get a linear system of four  homogenous
equations with unknowns $C_1$, $C_2$, $C_3$, $C_4$ and with parameter $\lambda$.
The eigenvalues of $\widehat{\mathcal{L}}(\lambda)$ are zeros of the
determinant $D(\lambda)$ of the matrix corresponding to the
system mentioned above. Omitting numerous technicalities
the resulting determinant $D(\lambda)$
reads as follows:

\begin{equation}\label{detDN}
D(\lambda)= \left| \begin{array}{cccc}
0&4-{i}\lambda&0&-2+{i}\lambda\\
\\
2+{i}\lambda&0&4+{i}\lambda&0\\
\\
d_{31}&d_{32}&d_{33}&d_{34}\\
\\
d_{41}&d_{42}&d_{43}&d_{44}
\end{array}\right|=0,
\end{equation}
where
\begin{displaymath}
\begin{array}{ll}
d_{31}=\cos( \frac{i\lambda\pi}{2})-
\frac{{i}\lambda}{2}\cos[({i}\lambda-2)\frac{\pi}{2}],&
d_{41}=\frac{{i} \lambda}{2} \sin[({i}\lambda-2) \frac{\pi}{2}],\\
\\
d_{32}=\sin( \frac{{i}\lambda\pi}{2})-
\frac{{i}\lambda}{2}\sin[({i}\lambda-2)\frac{\pi}{2}],&
d_{42}=-\frac{{i}
\lambda}{2}\cos[({i}\lambda-2)\frac{\pi}{2}],\\
\\
d_{33}=-\frac{{i} \lambda}{2} \cos[({i}\lambda-2)\frac{\pi}{2}],&
d_{43}=\sin( \frac{{i} \lambda\pi}{2})+
\frac{{i}\lambda}{2}\sin[({i}\lambda-2)\frac{\pi}{2}],\\
\\
d_{34}=\frac{{i} \lambda}{2} \sin[({i}\lambda-2)\frac{\pi}{2}],&
d_{44}=\cos( \frac{{i} \lambda\pi}{2}) +
\frac{{i}\lambda}{2}\cos[({i}\lambda-2)\frac{\pi}{2}].
\end{array}
\end{displaymath}
Computation of \eqref{detDN} leads to the transcendent equation
\begin{equation}\label{determinant DN}
({i}\lambda)^2  -
4\cos^2\left[({i}\lambda)\frac{\pi}{2}\right] - \sin^2
\left[({i}\lambda)\frac{\pi}{2}\right]=0.
\end{equation}
The roots of the equation (\ref{determinant DN})  are the eigenvalues of
$\widehat{\mathcal{L}}(\lambda)$.

\begin{remark}\label{roots}
We show that there exists $\varepsilon>0$ such that there are no eigenvalues  of
$\widehat{\mathcal{L}}(\lambda)$  situated in the strip ${\rm
Im}\lambda\in [-1-\varepsilon,0)$  with the exception of $\lambda = -{i}$.
It is easy to see that eigenvalue $\lambda = -{i}$ is simple.

Let us briefly present the technical procedure.
Let $\lambda = a + {i}b$, where $a$ and $b$ are real numbers.
Separating real and imaginary parts in
(\ref{determinant DN})
we get the following system of nonlinear equations
\begin{eqnarray}
(b^2-a^2) -\frac{5}{2}
&=&
\frac{3}{4}\cos(\pi b)\left( e^{\pi a}+e^{- \pi a} \right),
\label{reg stok dn 1}
\\
-2ab
&=&
\frac{3}{4} \sin(\pi b)\left(e^{\pi a}-e^{-\pi a}\right).
\label{reg stok dn}
\end{eqnarray}
The equation (\ref{reg stok dn}) can be simply modified to a more convenient form
\begin{equation}\label{pom20}
\frac{- \pi a}{ e^{\pi a}-e^{- \pi a}}
\left[\frac{2}{\pi} \right]^2
=
\frac{3}{2}
\frac{\sin(\pi b)}{\pi b},
\quad
a\neq0, \, b\neq 0.
\end{equation}
It is easy to see that the expression on the left is negative for every
$a \in \mathbb{R}$, $a\neq 0$, and the expression on the right is nonnegative
for every $b \in [-1,0)$. Consequently, corresponding $\lambda$ are not roots of (\ref{determinant DN}).

For $\lambda = {i}\,b$ the left hand side
of (\ref{determinant DN}) takes the form $b^2-1-3\cos^2(\frac{\pi\,b}{2})$. This expression is
negative for $b\in(-1,0)$. Therefore, $\lambda = {i}\,b$, where $b\in(-1,0)$, are not roots
of (\ref{determinant DN}).

The equation (\ref{reg stok dn 1}) can be written in the form
\begin{equation}\label{pom21}
\frac{(b^2-a^2)}{\left( e^{\pi a}+e^{- \pi a} \right)}
-\frac{5}{2 \left( e^{\pi a}+e^{- \pi a} \right)}
=
\frac{3}{4}\cos(\pi b).
\end{equation}
It is easy to see that there exist $\varepsilon_1>0$ and  $K>0$ such that
(\ref{pom21}) does not hold for  $|a|>K$ and
$b\in(-1-\varepsilon_1,-1)$. Hence, corresponding $\lambda$ are not roots of (\ref{determinant DN}).

Roots if (\ref{determinant DN}) are isolated points since the left hand side
corresponds to a nonzero holomorphic function defined on the whole C. Consequently,
$\lambda = a + {i}\,b$, where $|a|\leq K$ and $b\in(-1-\varepsilon,-1)$ for sufficiently
small $\varepsilon>0$, $0<\varepsilon<\varepsilon_1$,  are not roots of (\ref{determinant DN}).

All of these facts imply that there are no eigenvalues  of
$\widehat{\mathcal{L}}(\lambda)$  situated in the strip ${\rm
Im}\lambda\in [-1-\varepsilon,0)$  with the exception of $\lambda = -{i}$.
\end{remark}

Let $\varepsilon$ mentioned in Remark \ref{roots} be fixed. .
Let $(\bfvartheta,q)$ be a weak solution of
\eqref{stationary S 2D}--\eqref{neumann S} with the right hand side $\bfsigma$.
Suppose additionally $\bfsigma\in L^{2+\varepsilon}(\om)^2$.The cut-off
function $\eta(r)$ and number $\delta$ was defined in \eqref{cut_of_function}.
Let $\om_\delta=\om\cap U_\delta(A)$, $\tau<\delta/2$ be fixed
and $\om_A=\om\cap U_\tau(A)$. Theorem
\ref{theorem_kozmazross} yields the following asymptotic
representation in a neighborhood $\mathcal{O}$ of the corner point
$A$
\begin{multline}\label{general_solution_11}
\eta(r)\left(
\begin{array}{r}
\bfvartheta\\
q
\end{array}\right)
= \left[ c_1 \left(
\begin{array}{c}
r\cos\omega\\
-r\sin\omega\\
0\\
\end{array}\right) +
c_2 \left(
\begin{array}{c}
r\sin\omega\\
r\cos\omega\\
0\\
\end{array}\right)
\right. \\ \left.
+ c_3 \left(
\begin{array}{c}
-r\cos\omega\\
r\sin\omega\\
-4\\
\end{array}\right)
+c_4 \left(
\begin{array}{c}
-r\sin\omega\\
3r\cos\omega\\
0\\
\end{array}\right)
\right] + \left(
\begin{array}{r}
\bfvartheta_{reg}
\\
q_{reg}
\end{array}\right)
\end{multline}
with some constants $c_1$, $c_2$, $c_3$ and $c_4$, where
$(\bfvartheta_{reg},q_{reg})
\in W^{2,2+\varepsilon}(\om_\delta)^2\times
W^{1,2+\varepsilon}(\om_\delta)$. Now \eqref{general_solution_11}
immediately yields $\eta(r)\,(\bfvartheta,q) \in
W^{2,2+\varepsilon}(\om_\delta)^2\times
W^{1,2+\varepsilon}(\om_\delta)$. Since the line ${\rm Im}\lambda =
-1-\varepsilon$ is free of eigenvalues of the pencil operator
$\widehat{\mathcal{L}}(\lambda)$ then
\begin{equation}\label{est_lin_stokes_+eps}
\|\bfvartheta\|_{W^{2,2+\varepsilon}(\om_A)^2} +
\|q\|_{W^{1,2+\varepsilon}(\om_A)}\; \leq \cn09 \,
\|\bfsigma\|_{L^{2+\varepsilon}(\Omega)^2},
\end{equation}
where $\tau\cc09 = \cc09(\om_A)$.

Since $(\bfvartheta,q)\in W^{2,2-\varepsilon}(\om_\delta)^2\times
W^{1,2-\varepsilon}(\om_\delta)$ and the strip ${\rm
Im}\lambda\in [-1+\varepsilon,0)$ is free of eigenvalues of the
pencil operator $\widehat{\mathcal{L}}(\lambda)$ then
\begin{equation}\label{est_lin_stokes_-eps}
\|\bfvartheta\|_{W^{2,2-\varepsilon}(\om_A)^2} +
\|q\|_{W^{1,2-\varepsilon}(\om_A)}\; \leq \cn10 \,
\|\bfsigma\|_{L^{2-\varepsilon}(\Omega)^2},
\end{equation}
where $\cc10 = \cc10(\om_A)$.

Let $T_i$, $i = 1,2$, be operators such that
$$T_i(\bfsigma)\,:=\Bigl(\,\bfvartheta,
\frac{\partial\bfvartheta}{\partial x_1},
\frac{\partial\bfvartheta}{\partial x_2},
\frac{\partial^2\bfvartheta}{\partial{x_1}^2},
\frac{\partial^2\bfvartheta}{\partial {x_1}\partial{2_2}},
\frac{\partial^2\bfvartheta}{\partial {x_2}^2},
\frac{\partial q}{\partial x_1},
\frac{\partial q}{\partial x_2}\Bigr)$$
and $T_1$ and $T_2$, respectively, are defined on $L^{2-\varepsilon}(\Omega)$
and $L^{2+\varepsilon}(\Omega)$.
By (\ref{est_lin_stokes_+eps}) and
(\ref{est_lin_stokes_-eps})
$T_1:L^{2-\varepsilon}(\Omega)\to L^{2-\varepsilon}(\om_A)^6$,
$T_2:L^{2+\varepsilon}(\Omega)\to L^{2+\varepsilon}(\om_A)^6$
and both of them are continuous. Interpolating between them
(see \cite[Theorem 2.4]{Du}) we obtain the following theorem.

\begin{theorem}\label{theorem pom 2}
Let $\bfsigma \in L^2(\Omega)^2$ and $(\bfvartheta,q)\in
V_{\kappa}\times L^2(\Omega)$ be the solution of
\eqref{stationary S 2D}--\eqref{neumann S}. Let $A\in\partial\Omega$
be the boundary point where the boundary conditions change their
type. Then there exists $\om_A = U_\tau(A)\cap\om$ for
sufficiently small $\tau$ such that $(\bfvartheta,q)\in
W^{2,2}(\om_A)^2\times W^{1,2}(\om_A)$ and the
estimate \phantom{$\cn06$}
\begin{equation}\label{est_lin_stokes}
\|\bfvartheta\|_{W^{2,2}(\om_A)^2} +
\|q\|_{W^{1,2}(\om_A)}\; \leq \cc06 \,
\|\bfsigma\|_{L^2(\Omega)^2}
\end{equation}
holds with $\cc06 = \cc06(\om_A)$.
\end{theorem}

Proof of Theorem \ref{maintheorap}follows at once from Theorem
\ref{theorem pom 2}, (\ref{eq23b}), (\ref{estimate_ADN1}),
(\ref{estimate_ADN2}) and compactness of $\overline{\om}$.


\section*{Acknowledgement}
{The research was supported by the Grant Agency of the Czech Republic,  grant No. 13-18652S (author one)
and grant No. 13-00522S (author two).
}

\vspace{\baselineskip}


\end{document}